\documentclass[12pt,reqno]{amsart}
\usepackage{amsmath}
\usepackage{amsfonts}

\usepackage{mathrsfs}
\usepackage {amssymb,euscript}
\usepackage {amsmath}
\usepackage {amsthm}
\usepackage {amscd}
\usepackage{geometry}
\usepackage{hyperref}
\usepackage{color}
\usepackage{epsfig}
\usepackage{caption}

\usepackage{tikz}
\usepackage{tkz-euclide}
\usetkzobj{all}

\usepackage{graphics}

\theoremstyle{definition}
\newtheorem{lemma}{Lemma}[section]
\newtheorem{definition}[lemma]{Definition}
\newtheorem{proposition}[lemma]{Proposition}
\newtheorem{theorem}[lemma]{Theorem}

\newtheorem{remark}{Remark}

\numberwithin{equation}{section}
\newtheorem{innercustomthm}{Theorem}
\newenvironment{customthm}[1]
{\renewcommand\theinnercustomthm{#1}\innercustomthm}
{\endinnercustomthm}
\begin{document}

\def\RR{\mathbb{R}}
\def\ZZ{\mathbb{Z}}

\newcommand{\abs}[1]{\left\vert#1\right\vert}
\newcommand{\ub}{\underline{u}}
\newcommand{\Cb}{\underline{C}}
\newcommand{\Lb}{\underline{L}}
\newcommand{\Lh}{\hat{L}}
\newcommand{\Lbh}{\hat{\Lb}}
\newcommand{\phib}{\underline{\phi}}
\newcommand{\Phib}{\underline{\Phi}}
\newcommand{\Db}{\underline{D}}
\newcommand{\Dh}{\hat{D}}
\newcommand{\Dbh}{\hat{\Db}}
\newcommand{\omb}{\underline{\omega}}
\newcommand{\omh}{\hat{\omega}}
\newcommand{\ombh}{\hat{\omb}}
\newcommand{\Pb}{\underline{P}}
\newcommand{\chib}{\underline{\chi}}
\newcommand{\chih}{\hat{\chi}}
\newcommand{\chibh}{\hat{\chib}}

\newcommand{\alb}{\underline{\alpha}}
\newcommand{\zeb}{\underline{\zeta}}
\newcommand{\beb}{\underline{\beta}}
\newcommand{\etb}{\underline{\eta}}
\newcommand{\Mb}{\underline{M}}
\newcommand{\oth}{\hat{\otimes}}


\def\a {\alpha}
\def\b {\beta}
\def\ab {\alphab}
\def\bb {\betab}
\def\nab {\nabla}

\def\ub {\underline{u}}
\def\th {\theta}
\def\Lb {\underline{L}}
\def\Hb {\underline{H}}
\def\chib {\underline{\chi}}
\def\chih {\hat{\chi}}
\def\chibh {\hat{\underline{\chi}}}
\def\omegab {\underline{\omega}}
\def\etab {\underline{\eta}}
\def\betab {\underline{\beta}}
\def\alphab {\underline{\alpha}}
\def\Psib {\underline{\Psi}}
\def\hot{\widehat{\otimes}}
\def\Phib {\underline{\Phi}}
\def\thb {\underline{\theta}}
\def\t {\tilde}
\def\st {\tilde{s}}

\def\rhoc{\check{\rho}}
\def\sigmac{\check{\sigma}}
\def\Psic{\check{\Psi}}
\def\kappab{\underline{\kappa}}
\def\betabc {\check{\underline{\beta}}}

\def\d {\delta}
\def\f {\frac}
\def\i {\infty}
\def\l {\bigg(}
\def\r {\bigg)}
\def\S {S_{u,\underline{u}}}
\def\o{\omega}
\def\be{\begin{equation}\begin{split}}
\def\en{\end{split}\end{equation}}
\def\at{a^{\frac{1}{2}}}
\def\af{a^{\frac{1}{4}}}
\def\od{\omega^{\dagger}}
\def\ombd{\underline{\omega}^{\dagger}}
\def\K{K-\frac{1}{|u|^2}}
\def\ut{\frac{1}{|u|^2}}
\def\s{\frac{\delta a^{\frac{1}{2}}}{|u|}}
\def\Kb{K-\frac{1}{(u+\underline{u})^2}}
\def\bf{b^{\frac{1}{4}}}
\def\bt{b^{\frac{1}{2}}}
\def\de{\delta}
\def\ls{\lesssim}
\def\om{\omega}
\def\Om{\Omega}

\newcommand{\e}{\epsilon}
\newcommand{\et} {\frac{\epsilon}{2}}
\newcommand{\ef} {\frac{\epsilon}{4}}
\newcommand{\LH} {L^2(H_u)}
\newcommand{\LHb} {L^2(\underline{H}_{\underline{u}})}
\newcommand{\M} {\mathcal}
\newcommand{\TM} {\tilde{\mathcal}}
\newcommand{\p}{\psi\hspace{1pt}}
\newcommand{\q}{\underline{\psi}\hspace{1pt}}
\newcommand{\Li}{_{L^{\infty}(S_{u,\underline{u}})}}
\newcommand{\Lt}{_{L^{2}(S)}}
\newcommand{\da}{\delta^{-\frac{\epsilon}{2}}}
\newcommand{\db}{\delta^{1-\frac{\epsilon}{2}}}
\newcommand{\D}{\Delta}


\renewcommand{\div}{\mbox{div }}
\newcommand{\curl}{\mbox{curl }}
\newcommand{\trchb}{\mbox{tr} \chib}
\def\trch{\mbox{tr}\chi}
\newcommand{\tr}{\mbox{tr}}

\newcommand{\Ls}{{\mathcal L} \mkern-10mu /\,}
\newcommand{\eps}{{\epsilon} \mkern-8mu /\,}

\newcommand{\xib}{\underline{\xi}}
\newcommand{\psib}{\underline{\psi}}
\newcommand{\rhob}{\underline{\rho}}
\newcommand{\thetab}{\underline{\theta}}
\newcommand{\gammab}{\underline{\gamma}}
\newcommand{\nub}{\underline{\nu}}
\newcommand{\lb}{\underline{l}}
\newcommand{\mub}{\underline{\mu}}
\newcommand{\Xib}{\underline{\Xi}}
\newcommand{\Thetab}{\underline{\Theta}}
\newcommand{\Lambdab}{\underline{\Lambda}}
\newcommand{\vphb}{\underline{\varphi}}

\newcommand{\ih}{\hat{i}}

\newcommand{\tcL}{\widetilde{\mathscr{L}}}

\newcommand{\sRic}{Ric\mkern-19mu /\,\,\,\,}
\newcommand{\sL}{{\cal L}\mkern-10mu /}
\newcommand{\sLh}{\hat{\sL}}
\newcommand{\sg}{g\mkern-9mu /}
\newcommand{\seps}{\epsilon\mkern-8mu /}
\newcommand{\sd}{d\mkern-10mu /}
\newcommand{\sR}{R\mkern-10mu /}
\newcommand{\snab}{\nabla\mkern-13mu /}
\newcommand{\sdiv}{\mbox{div}\mkern-19mu /\,\,\,\,}
\newcommand{\scurl}{\mbox{curl}\mkern-19mu /\,\,\,\,}
\newcommand{\slap}{\mbox{$\triangle  \mkern-13mu / \,$}}
\newcommand{\sGamma}{\Gamma\mkern-10mu /}
\newcommand{\somega}{\omega\mkern-10mu /}
\newcommand{\somb}{\omb\mkern-10mu /}
\newcommand{\spi}{\pi\mkern-10mu /}
\newcommand{\sJ}{J\mkern-10mu /}
\renewcommand{\sp}{p\mkern-9mu /}
\newcommand{\su}{u\mkern-8mu /}

\title[Trapped Surface Formation in Spherical Symmetry]{Trapped surface formation for spherically symmetric Einstein-Maxwell-charged scalar field system with double null foliation}

\date{\today}

\author{Xinliang An$^*$$^1$}\author{Zhan Feng Lim$^2$}
\address{$^1$\small Department of Mathematics, National University of Singapore,
10 Lower Kent Ridge Road, Singapore, 119076}
\email{matax@nus.edu.sg}

\address{$^2$\small Department of Mathematics, National University of Singapore,
10 Lower Kent Ridge Road, Singapore, 119076}
\email{zflim@u.nus.edu}


\begin{abstract}
{\color{black} In this paper, under spherical symmetry we prove a trapped surface formation criterion for the Einstein-Maxwell-charged scalar field system. We generalize an approach introduced by Christodoulou for studying the Einstein-scalar field. In appendix, with double null foliation we also reprove Christodoulou's result and for Minkowskian incoming characteristic initial data, we improve Christodoulou's bound.  
}
\end{abstract}

\maketitle

\section{Introduction}

{\color{black}

\subsection{Motivation}
{\color{black}
In} a series of papers \cite{Chr.1}-\cite{Chr.4}, Christodoulou studied singularity formation for \textcolor{black}{the} Einstein-scalar field system:
\begin{equation}\label{ES}
\begin{split}
&\mbox{Ric}_{\mu\nu}-\f12Rg_{\mu\nu}={\color{black}8\pi      } T_{\mu\nu},\\
&T_{\mu\nu}=\partial_{\mu}\phi \partial_{\nu}\phi-\f12g_{\mu\nu}\partial^{\sigma}\phi \partial_{\sigma}\phi. 
\end{split}
\end{equation}

\noindent \textcolor{black}{Through these papers}, Christodoulou proved {\color{black}in four steps} that \textcolor{black}{under spherical symmetry, the} \textit{weak cosmic censorship conjecture} holds. {\color{black} More precisely, Christodoulou proved that for (\ref{ES}) with} large initial data, {\color{black} a so-called naked singularity} may form{\color{black}; however,} for generic initial data, these singularities {\color{black} are}  covered by a black hole region and are invisible for observers far away. 
{\color{black}These are celebrated results.}

The Penrose diagram of a spherically symmetric gravitational collapse spacetime {\color{black}for (\ref{ES}) with generic initial data} is as {\color{black}{follows}}: 

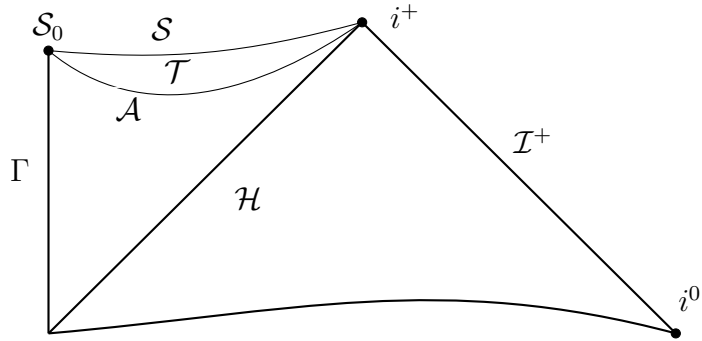
\begin{center}
\begin{minipage}[!t]{0.4\textwidth}
\begin{tikzpicture}[scale=0.75]
\draw [white](-1, -2.5)-- node[midway, sloped, above,black]{$\Gamma$}(0, -2.5);
\draw [white](0, 0)-- node[midway, sloped, above,black]{$\mathcal{S}$}(4, 0);
\draw [white](0, -0.75)-- node[midway, sloped, above,black]{$\mathcal{T}$}(4.5, -0.75);
\draw [white](-1, 0)-- node[midway, sloped, above,black]{$\mathcal{S}_0$}(1, 0);
\draw [white](5.5, 0.2)-- node[midway, sloped, above,black]{$i^+$}(7, 0.2);
\draw [white](10, -4.8)-- node[midway, sloped, above,black]{$i^0$}(12.5, -4.8);
\draw (0,0) to [out=-5, in=195] (5.5, 0.5);
\draw (0,0) to [out=-40, in=215] (5.5, 0.5);
\draw [white](0, -3)-- node[midway, sloped, above,black]{$\mathcal{H}$}(7, -3);
\draw [white](7, -2)-- node[midway, sloped, above,black]{$\mathcal{I^+}$}(10, -2);
\draw [white](0, -0.65)-- node[midway, sloped, below,black]{$\mathcal{A}$}(2.8, -0.65);

\draw [thick] (0, -5)--(0,0);
\draw [thick] (5.5, 0.5)--(0,-5);
\draw[fill] (0,0) circle [radius=0.08];
\draw[fill] (5.5, 0.5) circle [radius=0.08];
\draw[fill] (11, -5) circle [radius=0.08];
\draw [thick] (5.5, 0.5)--(11,-5);
\draw [thick] (0,-5) to [out=5, in=165] (11, -5);
\end{tikzpicture}
\end{minipage}
\begin{minipage}[!t]{0.6\textwidth}
\end{minipage}
\hspace{0.05\textwidth}
\end{center}
\noindent Here{\color{black}{,}} $\Gamma$ is the center of symmetry$-$invariant under $SO(3)${\color{black}, and} $i^+, \mathcal{I}^+, i_0$ are timelike infinity, future null infinity{\color{black}{,}} and spacelike infinity {\color{black}respectively}. The boundary of the causal past of $i^+$ is $\mathcal{H}$, {\color{black}which} is called {\color{black}the} event horizon. $\mathcal{T}$ is the trapped region, where {\color{black}not even light can} escape to $\mathcal{I^+}$. $\mathcal{A}$ is called {\color{black}the} apparent horizon and it is the lower boundary of $\mathcal{T}$. $\mathcal{S}_0$ is the first singular point along $\Gamma$ and $\mathcal{S}$ is the singular boundary of $\mathcal{T}$.

A crucial step of Christodoulou's {\color{black}proof of the weak cosmic censorship conjecture} is \cite{Chr.1}. There, Christodoulou established a sharp trapped surface\footnote{A trapped surface is a two-dimensional sphere, with both incoming and outgoing null expansions negative.} formation criterion for (\ref{ES}). Christodoulou's original proof in \cite{Chr.1} was based on a geometric Bondi coordinate {\color{black}system} with a null frame. 

{\color{black}However, at present, the double null foliation is a more popular choice of coordinate system.} {\color{black}There have been many recent works published in general relativity using a double null foliation}. {\color{black}In order to generalize Christodoulou's results} in \cite{Chr.1}-\cite{Chr.4} to other matter {\color{black}models}, here we {\color{black}adopt the} double null foliation. In our paper, we {\color{black}will review Christodoulou's result in the setting of a double null foliation. Then,} we {\color{black}will} generalize his result to {\color{black}the} Einstein-Maxwell-{\color{black}charged }scalar field system. 

{\color{black}
Within {\color{black}the study of spherically symmetric systems}, there are interesting results {\color{black}on formation of trapped surfaces and singularities} for other matter {\color{black}models}, e.g{\color{black}.} Einstein-Vlasov studied by Andr\'easson \cite{And}, And\'easson-Rein\cite{AR},  Moschidis\cite{Mo},  Einstein-Euler studied by Burtscher and LeFloch \cite{BL}, Einstein-scalar field studied by Li-Liu \cite{LL}, An-Zhang \cite{AZ}, An-Gajic \cite{AG}, Einstein-null dust studied by Moschidis\cite{Mo2}, Einstein-scalar field with positive cosmological constant by Costa \cite{JC}. For {\color{black}the} Einstein-Maxwell-(real) scalar field system, {\color{black}we refer interested readers} to \cite{Da1, Da2} by Dafermos, \cite{LO} by Luk and Oh on the recent development of proving strong cosmic censorship. {\color{black}And for the Einstein-Maxwell-charged scalar field system}, we refer to \cite{VDM}-\cite{VDM3} by Van de Moortel.} 

\subsection{The Main Result}

{\color{black}We consider the characteristic initial value problem for (\ref{ES}) in the {\color{black}rectangular} region.}

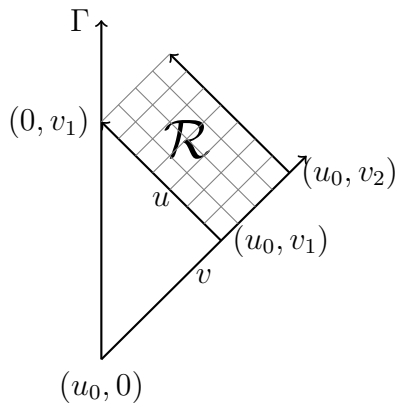
\begin{minipage}[!t]{0.4\textwidth}
	\begin{tikzpicture}[scale=0.9]
	\node[] at (1.25,3.25) {\LARGE $\mathcal{R}$};
	\begin{scope}[thick]
	\draw[->] (0,0) node[anchor=north]{$(u_0,0)$} -- (0,5)node[anchor = east]{$\Gamma$};
	\draw[->] (0,0) --node[anchor = north]{$v$} (3,3);
	\draw[->] (1.75,1.75) node[anchor=west]{$(u_0,v_1)$} --node[anchor=north]{$u$} (0,3.5)node[anchor = east]{$(0,v_1)$};
	\draw[->] (2.75,2.75) node[anchor=west]{$(u_0,v_2)$} -- (1,4.5);
	\end{scope}
	\begin{scope}[gray]
	\draw (2,2) -- (0.25,3.75);
	\draw (2.25,2.25) -- (0.5,4);
	\draw (2.5,2.5) -- (0.75,4.25);
	\draw(1.5,2) -- (2.5,3);
	\draw(1.25,2.25) -- (2.25,3.25);
	\draw(1,2.5) -- (2,3.5);
	\draw(0.75,2.75) -- (1.75,3.75);
	\draw(0.5,3) -- (1.5,4);
	\draw(0.25,3.25) -- (1.25,4.25);
	\draw(0,3.5) -- (1,4.5);
	\end{scope}
	\end{tikzpicture}
\end{minipage}
\begin{minipage}[!t]{0.58\textwidth}
We {\color{black}employ} the double-null foliation {\color{black}with} $u$ and $v$ {\color{black}as}  optical functions{\color{black}; that is,} {\color{black} $g^{\a\b}\partial_{\a}u\partial_{\b}u=0$ and $g^{\a\b}\partial_{\a}v\partial_{\b}v=0$. Thus, we have} $u=\mbox{{\color{black}constant}}$  {\color{black}as} the outgoing null hypersurface; $v=\mbox{{\color{black}constant}}$ {\color{black}as} the incoming null hypersurface. \\

\noindent {\color{black}Due to} spherical symmetry, {\color{black}we have a central axis $\Gamma$.} We prescribe initial data along {\color{black}the }outgoing cone $u=u_0$ and {\color{black}the} incoming cone $v=v_1$.
\end{minipage}
\hspace{0.05\textwidth}

\noindent For the metric of the $3+1$-dimensional spacetime, we {\color{black}impose spherical symmetry and write it with double-null coordinates:} 
\begin{equation}\label{metric0}
g_{\mu\nu}dx^{\mu}dx^{\nu}=-\Omega^2(u,v)dudv+r^2(u,v)\big(d\theta^2+\sin^2\theta d\phi^2\big).
\end{equation}
\noindent In {\color{black}the} above diagram every point $(u,v)$ {\color{black}represents} a $2$-sphere $S_{u,v}$. {\color{black}The Hawking mass of such a 2-sphere} is defined as 
\begin{equation}\label{Hawking mass}
m(u,v)=\f{r}{2}(1+4\Omega^{-2}\partial_u r \partial_v r).
\end{equation}
\noindent Along $u=u_0$, we also define {\color{black}{the}} initial mass input
$$\eta_0:=\f{m(u_0, v_2)-m(u_0, v_1)}{r(u_0, v_2)}, \mbox{ and denote } \d_0:=\f{r(u_0, v_2)-r(u_0, v_1)}{r(u_0, v_2)}.$$
{\color{black}Finally, let  $u_*$ denote the value of $u\in[u_0,u_*]$ such that ${\color{black}r(u_*, v_2)}=\frac{3\delta_0}{1+\delta_0}\cdot r(u_0, v_2)$.}

\begin{theorem}{\textcolor{black}{(Christodoulou \cite{Chr.1} and reproved in appendix)}}\label{thm1.1}\\
	{\color{black}Define the function}
	\begin{align*}
	E(x):=\frac{x}{(1+x)^2}\bigg[\ln\bigg(\frac{1}{2x}\bigg)+5-x\bigg].
	\end{align*}
{\color{black}Consider the system (\ref{ES}) with characteristic initial data along $u=u_0$ and $v=v_1$.} For initial mass input $\eta_0$ along $u=u_0$,  {\color{black}if the following lower bound holds:}
	\begin{align*}
	\eta_0>E(\delta_0),
	\end{align*}
		then a trapped surface {\color{black}$S_{u,v}$, with properties $\partial_v r(u,v)<0$} and $\partial_u r(u,v)< 0$, {\color{black}forms} in {\color{black}the region $[u_0,u_*]\times[v_1,v_2]\subset\mathcal{R}$} .
		\end{theorem}
\begin{remark}For $0<\d_0\ll 1$, {\color{black}we can check that the order of the lower bound of $\eta_0$, $E(\d_0)$, is of order $\d_0\ln(\f{1}{\d_0})$. Hence,} if $\eta_0\gtrsim\delta_0\ln\bigg(\frac{1}{\delta_0}\bigg)$, {\color{black}a trapped surface is guaranteed to form within $\mathcal{R}$}.
\end{remark}

\noindent The above theorem is crucial for Christodoulou's final proof of {\color{black}the} weak cosmic censorship in \cite{Chr.4}. There{\color{black},} Christodoulou studied the first singular point formed in {\color{black}the evolution of (\ref{ES})}: if that point is not covered by a trapped region, then a perturbation of the initial data would {\color{black}lead to} the condition in Theorem \ref{thm1.1} being satisfied{\color{black}. Hence,}  a trapped surface would form to cover that singular point. \\

\noindent \textcolor{black} {We provide a reproof of Theorem \ref{thm1.1} in the appendix. While Christodoulou's proof was written in Bondi coordinates, here we have rewritten it in a double null foliation.} Double null foliations are widely used in studying both the exterior and interior regions of black holes for various matter models. Many results pertaining to spherical symmetry are also based on double null foliations. Hence, there is strong motivation to rewrite \cite{Chr.1} with a double null foliation.\\

\noindent \textcolor{black}{By strenghtening the hypothesis on the initial data in Theorem \ref{thm1.1}, in appendix we also improve Christodoulou's bound:}

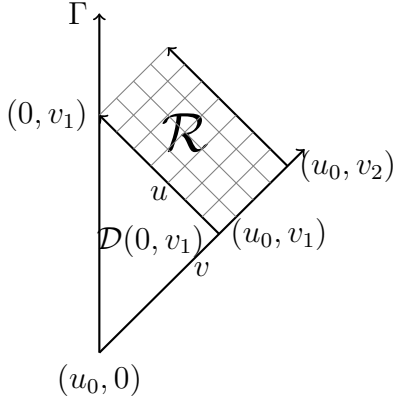
\begin{minipage}[!t]{0.4\textwidth}
	\begin{tikzpicture}[scale=0.9]
	\node[] at (1.25,3.25) {\LARGE $\mathcal{R}$};
	\node[] at (0.75,1.65) { $\mathcal{D}(0,v_1)$}; 
	\begin{scope}[thick]
	\draw[->] (0,0) node[anchor=north]{$(u_0,0)$} -- (0,5)node[anchor = east]{$\Gamma$};
	\draw[->] (0,0) --node[anchor = north]{$v$} (3,3);
	\draw[->] (1.75,1.75) node[anchor=west]{$(u_0,v_1)$} --node[anchor=north]{$u$} (0,3.5)node[anchor = east]{$(0,v_1)$};
	\draw[->] (2.75,2.75) node[anchor=west]{$(u_0,v_2)$} -- (1,4.5);
	\end{scope}
	\begin{scope}[gray]
	\draw (2,2) -- (0.25,3.75);
	\draw (2.25,2.25) -- (0.5,4);
	\draw (2.5,2.5) -- (0.75,4.25);
	\draw(1.5,2) -- (2.5,3);
	\draw(1.25,2.25) -- (2.25,3.25);
	\draw(1,2.5) -- (2,3.5);
	\draw(0.75,2.75) -- (1.75,3.75);
	\draw(0.5,3) -- (1.5,4);
	\draw(0.25,3.25) -- (1.25,4.25);
	\draw(0,3.5) -- (1,4.5);
	\end{scope}
	\end{tikzpicture}
\end{minipage}
\begin{minipage}[!t]{0.58\textwidth}
	{\color{black}\begin{theorem}\label{thm1.2}
		Assume that {\color{black}Minkowskian data are prescribed along $v=v_1$ and require $\phi(u,v_1)=0$}. Suppose that the following lower bound on $\eta_0$ holds:
		\begin{align*}
		\eta_0>\f92\delta_0,
		\end{align*}
		then there exist a MOTS {\color{black}or a trapped surface} in $[u_0,u_*]\times[v_1,v_2]\subset\mathcal{R}$, i.e. $\partial_vr\leq 0$ at some point in $[u_0,u_*]\times[v_1,v_2]$.
	\end{theorem}}
\end{minipage}
\hspace{0.05\textwidth}	

\begin{remark}
{\color{black}{Theorem \ref{thm1.2}}} improves the almost scale critical result in Theorem \ref{thm1.1}, i.e., $\eta_0>\d_0 \ln \bigg(\f{1}{\d_0}\bigg)$ {\color{black}implies} trapped surface formation, to a scale critical result, i.e., $\eta_0>\f92\d_0$ {\color{black}{implies trapped surface formation}}. {\color{black}In \cite{AL}, the first author and Luk first noted that by prescribing Minkowskian data along $v=v_1$, for Einstein vacuum equations a scale-critical trapped surface formation criterion could be established.\footnote{For more discussions about scaling consideration, interested readers are also refereed to \cite{An2012} and \cite{An2019} by the first author.} For $a$ being a large universal constant, the corresponding requirement for $\eta_0$ is $\eta_0\geq \d a$. For Einstein-scalar field system under spherical symmetry, Theorem 1.2 improves the large universal constant $a$ into a \underline{concrete} number $9/2$.} \end{remark}

\noindent The main result of our paper is the next theorem. We generalize  {\color{black}{the}} above results to {\color{black}{the}} Einstein scalar field coupled with the electromagnetic field. More precisely, we consider the following {\color{black}Einstein-Maxwell-charged scalar field} system:
\begin{gather*}
     R_{\mu\nu}-\frac{1}{2}g_{\mu\nu}R = 8\pi T_{\mu\nu},\\
     T_{\mu\nu}=T^{SF}_{\mu\nu}+T^{EM}_{\mu\nu},\\
      T^{SF}_{\mu\nu} =\frac{1}{2}D_\mu\phi(D_\nu\phi)^\dag+\frac{1}{2}D_\nu\phi(D_{\mu}\phi)^\dag-\frac{1}{2}g_{\mu\nu}\big(g^{\alpha\beta} D_\alpha\phi(D_\beta\phi)^\dag\big),\\
    T_{\mu\nu}^{EM}=\frac{1}{4\pi}\big(g^{\alpha\beta}F_{\alpha\mu}F_{\beta\nu}-\frac{1}{4}g_{\mu\nu}F^{\alpha\beta}F_{\alpha\beta}\big).
\end{gather*}
{\color{black}Here,} the Einstein scalar field is coupled to the electromagnetic field by the following {\color{black}{form of the}} Maxwell equation:
\begin{align*}
    \nabla^\nu F_{\mu\nu} = 2\pi\mathfrak{e}i\big(\phi(D_\mu\phi)^\dag-\phi^\dag D_\mu\phi\big),
\end{align*}
where $D_\mu:={\color{black}\partial_{\mu}+\mathfrak{e}iA_\mu}$ is known as the gauge covariant derivative. {\color{black}Here $\mathfrak{e}$ is the coupling constant and $A_{\mu}$ is the electromagnetic potential.} {\color{black} Using the gauge covariant derivative instead of the usual derivative ensures} that the physical equations remain invariant under local $U(1)$ transformations on $\phi$.  

{\color{black}Recall that under spherical symmetry, we have the ansatz (\ref{metric0}). {\color{black}Using the $\Omega$ appearing in the ansatz}, we define the charge $Q(u,v)$ contained in a sphere $S(u,v)$ to be $Q:=2r^2\Omega^{-2}F_{uv}$.}

\begin{theorem}{\color{black}(Main Theorem)}\label{thm1.3} 
 {\color{black}Denoting the outgoing null hypersurface $u=u_0$  by $C$ and {\color{black}the} incoming null hypersurface $v=v_1$ by $\underline{C}$, we define
$$\epsilon:=\sup_{C\cup\underline{C}}\frac{Q^2}{r^2}<1, \mbox{ and }L:=\sup_{\underline{C}}r|\phi|^2.$$}

{\color{black}\noindent Let $\o$ be \underline{any} positive constant in $(0,\frac{2}{3})$. Choose $v_2-v_1$ sufficiently small such that

\begin{gather}
\frac{9\mathfrak{e}^2}{4(1-\epsilon)^2}(v_2-v_1)^2+\frac{12\pi L\mathfrak{e}}{1-\epsilon}(v_2-v_1)\leq\frac{\omega}{4}\label{first assumption on v2-v1},\\
\frac{45\pi\mathfrak{e}^2(v_2-v_1)^2}{\pi(1-\epsilon)^2}+160\pi\mathfrak{e}^2{\color{black}r}(u_0, {\color{black}v_2})\frac{v_2-v_1}{1-\epsilon}|\phi_1|^2\leq 4\omega\label{second assumption on v2-v1}.
\end{gather}
}

{\color{black}\noindent Further require that the initial data along $\underline{C}$ are \underline{not} supercharged, i.e. 
\begin{equation}\label{non supercharged}
m(u,v_1)\geq |Q|(u,v_1).
\end{equation}}
\noindent Denote
	\begin{align*}
			g_\omega(x):=\frac{1+\frac{\omega}{2}}{1-\frac{\omega}{2}}\frac{1}{(1+x)^2}\bigg(\bigg(\frac{2^{1-\frac{\omega}{2}}}{\omega}+\frac{1}{2^{1+\frac{\omega}{2}}(1+\frac{\omega}{2})}\bigg)x^{1-\frac{\omega}{2}}-\frac{2}{\omega}x-\frac{1}{1+\frac{\omega}{2}}x^2\bigg).
	\end{align*} 

\noindent Assume the following lower bound on $\eta_0$ holds
	
	\begin{align*}
	\eta_0>\max\bigg\{\frac{13\epsilon}{\omega}+g_\omega(\delta_0),\frac{9}{2^{1+\frac{\omega}{2}}(1+\delta_0)^2}\delta_0^{1-\frac{\omega}{2}}+g_\omega(\delta_0)\bigg\},
	\end{align*}
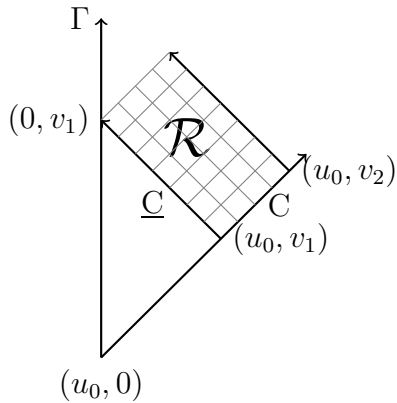
\begin{minipage}[!t]{0.4\textwidth}
	\begin{tikzpicture}[scale=0.9]
	\node[] at (1.25,3.25) {\LARGE $\mathcal{R}$};
	\node[] at (2.6,2.25) {C};
	\node[] at (0.75,2.25) {\underline{C}};
	\begin{scope}[thick]
	\draw[->] (0,0) node[anchor=north]{$(u_0,0)$} -- (0,5)node[anchor = east]{$\Gamma$};
	\draw[->] (0,0) --node[anchor = north]{$$} (3,3);
	\draw[->] (1.75,1.75) node[anchor=west]{$(u_0,v_1)$} --node[anchor=north]{} (0,3.5)node[anchor = east]{$(0,v_1)$};
	\draw[->] (2.75,2.75) node[anchor=west]{$(u_0,v_2)$} -- (1,4.5);
	\end{scope}
	\begin{scope}[gray]
	\draw (2,2) -- (0.25,3.75);
	\draw (2.25,2.25) -- (0.5,4);
	\draw (2.5,2.5) -- (0.75,4.25);
	\draw(1.5,2) -- (2.5,3);
	\draw(1.25,2.25) -- (2.25,3.25);
	\draw(1,2.5) -- (2,3.5);
	\draw(0.75,2.75) -- (1.75,3.75);
	\draw(0.5,3) -- (1.5,4);
	\draw(0.25,3.25) -- (1.25,4.25);
	\draw(0,3.5) -- (1,4.5);
	\end{scope}
	\end{tikzpicture}
\end{minipage}
\begin{minipage}[!t]{0.58\textwidth}

\noindent then a trapped surface {\color{black}is guaranteed to form in $[u_0,u_*]\times[v_1,v_2]\subset\mathcal{R}$}. 

\end{minipage}
\hspace{0.05\textwidth}

\end{theorem}
\begin{remark}
By comparing the order of the lower bounds (when {\color{black}$0<\delta_0\ll 1$}) of $\eta_0$ in the hypothesis of the theorem, we can interprete the theorem as: If $\eta_0\gtrsim\delta_0^{1-\frac{\omega}{2}}+\frac{13\epsilon}{\omega}$, a trapped surface {\color{black}forms in $\mathcal{R}$}. {\color{black}Since $\o$ could be chosen to be arbitrary small number in $(0,\f32)$, if we require that $\e$ (upper bound of $\frac{Q^2}{r^2}$ on $C\cup\underline{C}$) is small and satisfies $\f{13\e}{\o}\leq \d_0^{1-\f{\o}{2}}$, our theorem is also an \underline{almost-scale-critical} result.} 
\end{remark}

\begin{remark}
{\color{black}
Moreover, although we use the symbol $\epsilon$ to denote the upper bound of $\frac{Q^2}{r^2}$ on $C\cup\underline{C}$,  $\e$ is not necessarily to be small. In particular, we could choose $0<\d_0\ll1, \o=\f12$ and require $\e$ to be of size $1$. This is \underline{not} in the perturbative regime of Christodoulou's result for Einstein-scalar field.} Intuitively, for this case our Theorem \ref{thm1.3} is saying: \textit{if the incoming mass contained between $v_1$ and $v_2$ is large enough to overcome the initial charge on $C\cup\underbar{C}$, then we can guarantee the formation of a trapped surface}.
\end{remark}

{\color{black}
\begin{remark}\label{nonsupercharged} For initial data along $v=v_1$, we require that the initial data {\color{black}are not} super-charged, i.e. 
\begin{equation}\label{m Q 1}
m(u, v_1)\geq |Q|(u, v_1).
\end{equation} 
{\color{black}It is natural to consider initial data, which are not-super-charged, otherwise there could be non-physical super-charged naked singularity prescribed along $v=v_1$. At the same time, (\ref{m Q 1}) also implies an important inequality used in the proof of Proposition $\ref{mixed derivatives of r}$. 
\begin{lemma}
Along $v=v_1$, condition (\ref{m Q 1}) implies 
\begin{equation}\label{m Q 2}
\f{m}{r}(u,v_1)\geq \f{Q^2}{r^2}(u,v_1).
\end{equation}
\end{lemma}
}

\begin{proof} 
Since there is no MOTS or trapped surface along $v=v_1$, we have 
$$\f{2m}{r}(u,v_1)\leq 1, \mbox{ which gives } \f{m}{r}(u,v_1)\leq \f12.$$
Together with the non-super charged condition, we also have 
\begin{equation}\label{Q m}
\f{|Q|}{r}(u,v_1)\leq \f{m}{r}(u,v_1)\leq \f12.
\end{equation}
Then, we have
\begin{equation*}
\begin{split}
\f{m}{r}(u,v_1)-\f{Q^2}{r^2}(u,v_1)\geq& \f{|Q|}{r}(u,v_1)-\f{Q^2}{r^2}(u,v_1)\\
\geq& \f{|Q|}{r}(1-\f{|Q|}{r})(u,v_1)\geq 0.
\end{split}
\end{equation*}
{\color{black}
For the last inequality, we used (\ref{Q m})}.
\end{proof} 

{\color{black}The inequality \eqref{m Q 2}
is crucial in proving Proposition \ref{mixed derivatives of r}. And all subsequent results in Section \ref{main section} depend on Proposition \ref{mixed derivatives of r}.}
\end{remark}

\section{Preliminaries and Set-up}\label{preliminary}
To study the problem of trapped surface formation, we need to choose a convenient coordinate system {\color{black}in which} to express the {\color{black}Einstein field equations}. We describe the double null coordinate system for $\textit{spherically symmetric spacetimes}$ in {\color{black}what} follows.\\

\begin{definition}
	A spacetime $(\mathcal{M},g)$ is called \textit{spherically symmetric} if $SO(3)$ acts on it by isometry, and the orbits of the group are (topological) 2-dimensional spheres $S$. We define the area-radius coordinate $r(S)$ such that $A = 4\pi r^2$, where $A$ is the area of the $S$ determined by the induced metric $g|_S$.\footnote{This definition for $r$ implies that $g|_S = r^2(d\theta^2+sin^2\theta d\phi^2)$}.
\end{definition}

Under the assumption of spherical symmetry, a spacetime can be represented by {\color{black}a} two-dimensional diagram by considering only the quotient $\mathcal{M}/S$. Hence, a point on such diagram represents a 2-sphere in spacetime. There is no loss in generality by assuming that the outgoing ($v$ coordinate) and incoming ($u$ coordinate) null geodesics make 45-degree angles with the horizontal and vertical axes. Furthermore, it is possible to bring the points at infinity to a finite region through a conformal transformation, so that we can visualize the entire spacetime in a finite region. Such a representation of {\color{black}a} spacetime is called a \textit{Penrose diagram}.\\
{\color{black}
We now introduce the setup of the coordinate system, along with important points of interest on the Penrose diagram.
\begin{center}
	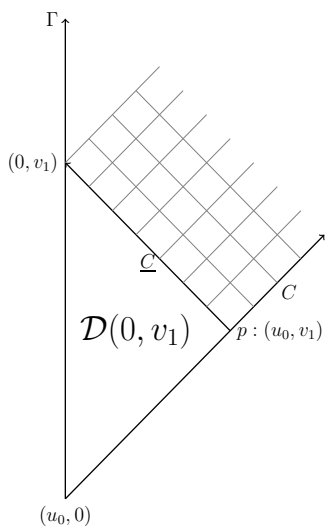
\begin{figure}
		\resizebox{4.5cm}{7cm}{
			\begin{tikzpicture}
			\node[] at (1.5,3.5) {\LARGE $\mathcal{D}(0,v_1)$};
			\begin{scope}[thick]
			\draw[->] (0,0) node[anchor=north]{$(u_0,0)$} -- (0,10)node[anchor = east]{$\Gamma$};
			\draw[->] (0,0) --node[anchor = north][label={[label distance=2cm]40:$C$}]{} (5.5,5.5);
			\draw[->] (3.5,3.5) node[anchor=west]{$p: (u_0,v_1)$} --node[anchor=north]{$\underline{C}$} (0,7)node[anchor = east]{$(0,v_1)$};
			\end{scope}
			\begin{scope}[gray]
			\draw (4,4) -- (0.5,7.5);
			\draw (4.5,4.5) -- (1,8);
			\draw (5,5) -- (1.5,8.5);
			\draw(3,4) -- (5,6);
			\draw(2.5,4.5) -- (4.5,6.5);
			\draw(2,5) -- (4,7);
			\draw(1.5,5.5) -- (3.5,7.5);
			\draw(1,6) -- (3,8);
			\draw(0.5,6.5) -- (2.5,8.5);
			\draw(0,7) -- (2,9);
			\end{scope}
			\end{tikzpicture}
		}\caption{Illustration of double null coordinate patch} 
	\label{fig: doublenull}
		
	\end{figure} 
\end{center}
\begin{enumerate}
	\item Let $\Gamma$ denote the axis of symmetry of the spacetime.
	\item Fix a point $p$ on the penrose diagram. Label the incoming null geodesic intersecting $p$ by $\underline{C}$, and the outgoing null geodesic intersecting $p$ by $C$. On the actual spacetime, $C$ and $\underline{C}$ are therefore null hypersurfaces.
	\item Parametrize $\underline{C}$ with the variable $u$, and $C$ by the variable $v$. At the intersection of $\Gamma$ and $\underline{C}$, set $u = 0$. Extend $C$ backwards until it intersects $\Gamma$. At the intersection of $\Gamma$ and $C$, we similarly set $v = 0$. Fixing these values determines the coordinate of $p$, which we call $(u_0,v)$.
	\item In the domain of dependence of $C\cup\underline{C}$, we can now establish a coordinate system: through every point in the domain of dependence runs an incoming and outgoing null geodesic emanating from $C$ and $\underline{C}$ respectively. Using the parameters $u$ and $v$ defined on $\underline{C}$ and $C$ gives us a coordinate for the point in question.
	\item Finally, let $\mathcal{D}(0,v_1)$ denote the region in spacetime bounded by $C$, $\underline{C}$ and $\Gamma$.
\end{enumerate}}
The construction above is illustrated in Figure \ref{fig: doublenull}. With respect to the double null coordinate system, the spherically symmetric metric can be expressed as 

\begin{align}
g=-\Omega^2(u,v)dudv + r^2(u,v)d\theta^2 +r^2(u,v)\sin^2\theta d\phi^2. \label{metric}
\end{align}

We now define several useful geometric quantities: 
\begin{definition}
	The \textit{Hawking mass} $m(u,v)$ contained inside a sphere $S(u,v)$ is defined to be the quantity {\color{black}$\frac{r}{2}(1+4\Omega^{-2}\partial_u r \partial_v r)$}. 
\end{definition}	

\begin{definition}
We define the charge $Q(u,v)$ contained in a sphere $S(u,v)$ to be
\begin{align}
Q:=2r^2\Omega^{-2}F_{uv}.\label{chargeeqn}
\end{align}
\end{definition}

\begin{gather}
\mbox{Note: } F_{uv}=\textcolor{black}{\partial_uA_v}-\partial_vA_u, \text{where $A$ is the electromagnetic potential.}\label{maxwelleqn}
\end{gather}
{\color{black}
	Due to gauge freedom in the electromagnetic potential, we can impose the condition $A_v\equiv 0$. Hence the above definition becomes:
	\begin{gather}
	F_{uv} = -\partial_vA_u\nonumber.
	\end{gather}
}
By substituting the expression {\color{black}$\eqref{metric}$,  $\eqref{chargeeqn}$ and $\eqref{maxwelleqn}$} into the Einstein field equations and Maxwell equations, we arrive at the following system of equations with dynamical real-valued unknowns $r, A_u$ and $\Omega^2$, and {\color{black}complex-valued unknown} $\phi$. For a more comprehensive explanation of these 
variables, we refer to \cite{Ko}, from which the following {\color{black}Einstein-Maxwell-charged scalar field system} has been obtained.

\begin{gather}
r\partial_v\partial_ur+\partial_vr\partial_ur=-\frac{\Omega^2}{4}\bigg(1-\frac{Q^2}{r^2}\bigg)\label{EMS1},\\
r^2\partial_u\partial_v\log\Omega = -2\pi r^2\big(D_u\phi(\partial_v\phi)^\dag+\partial_v\phi(D_u\phi)^\dag\big)-\frac{1}{2}\Omega^2\frac{Q^2}{r^2}+\frac{1}{4}\Omega^2+\partial_u\partial_v r\label{EMS2},\\
\partial_u(\Omega^{-2}\partial_ur) = -4\pi r\Omega^{-2}D_u\phi(D_u\phi)^\dag\label{EMS3},\\
\partial_v(\Omega^{-2}\partial_vr) = -4\pi r\Omega^{-2}\partial_v\phi(\partial_v\phi)^\dag\label{EMS4},\\
r\partial_u\partial_v\phi+\partial_ur\partial_v\phi+\partial_vr\partial_u\phi+\mathfrak{e}i\Psi(A)=0\label{EMS5},\\
\Psi(A) = A_u\partial_v(r\phi)-\frac{\Omega^2}{4}\frac{Q}{r}\phi\label{EMS6},\\
Q=-2r^2\Omega^{-2}\partial_vA_u\label{EMS7},\\
\partial_uQ=2\pi\mathfrak{e}ir^2\big(\phi(D_u\phi)^\dag-\phi^\dag D_u\phi\big)=4\pi\mathfrak{e}r^2Im\big(\phi^\dag D_u\phi\big)\label{EMS8},\\
\partial_vQ=2\pi\mathfrak{e}ir^2\big(\phi(\partial_v\phi)^\dag-\phi^\dag \partial_v\phi\big)=4\pi\mathfrak{e}r^2Im\big(\phi^\dag\partial_v\phi\big)\label{EMS9},
\end{gather}
where $D_u:= \partial_u+i\mathfrak{e}A_u$, and $\mathfrak{e}$ is the coupling constant between the scalar and electromagnetic field. It is worth noting that to reduce the above system into that of an uncharged scalar field, it suffices to set $\mathfrak{e} = 0$. Also, we can combine $\eqref{EMS5}$, $\eqref{EMS6}$ and $\eqref{EMS7}$, which gives us:

\begin{align}
&r\partial_u\partial_v\phi+\partial_ur\partial_v\phi+\partial_vr\partial_u\phi+\mathfrak{e}iA_u\partial_v(r\phi)-\mathfrak{e}i\frac{\Omega^2}{4}\frac{Q}{r}\phi=0\nonumber\\
&\implies \partial_v(r\partial_u\phi) + \partial_ur\partial_v\phi+\mathfrak{e}i\bigg(A_u\partial_v(r\phi)+r\phi\partial_vA_u\bigg)=-\mathfrak{e}i\frac{Q\phi\Omega^2}{4r}\nonumber\\
&\implies\partial_v(r\partial_u\phi)+\partial_ur\partial_v\phi+\mathfrak{e}i\partial_v(r\phi A_u)\nonumber\\&=\partial_v(rD_u\phi)+\partial_v\phi\partial_ur=-\mathfrak{e}i\frac{Q\phi\Omega^2}{4r}.\label{complex wave equation}
\end{align}

The above system \eqref{EMS1}-\eqref{complex wave equation} is subject to initial conditions. There are two types of initial conditions to be considered: 
\begin{enumerate}
	\item The first type of initial conditions are derived from geometrical considerations and are independent of the physical scenario. On the center of symmetry $\Gamma$, we must have $r=0$. In addition, by the spherical symmetry assumption, as we consider points infinitesimally close to the center, its incoming null geodesics essentially become outgoing (in the opposite direction). Hence we require that $\partial_vr(u_0,0) =-\partial_ur(u_0,0)$. The evolution of $r$ in the spacetime is then determined by equations $\eqref{EMS1}, \eqref{EMS3}, \text{ and } \eqref{EMS4}$.\\
	
	\textcolor{black}{On $C\cup\underline{C}$, we set $\Omega^2 = 1$.} This amounts to fixing a normalization for the coordinate system. The evolution of $\Omega^2$ in the coordinate patch $[u_0,0]\times[v_1,\infty)$ is then given by equation $\eqref{EMS2}$.\\
	\item The second type of initial conditions are those derived from quantities such as the scalar field $\phi$ and electromagnetic potential $A_u$. We can prescribe initial data of $\phi$ {\color{black}freely} on $C\cup\underline{C}$, which will completely determine its first derivatives as $C$ and $\underline{C}$ are characteristic hypersurfaces.\\  
	
	The electromagnetic potential $A_u$ along outgoing null hypersurfaces can be determined through equation $\eqref{EMS7}$ up to an arbitrary constant, which is in turn determined by $\eqref{EMS9}$. For completeness, it is worth mentioning that there is no loss in generality in letting $A_u = 0$ along $\Gamma$ \textcolor{black}{due to gauge freedom}, although we will not make use of this fact.\\
\end{enumerate}

Using the above system of equations, we can compute the derivatives of the Hawking mass:
\begin{align}
\partial_um&= \partial_u\bigg(\frac{r}{2}(1+4\Omega^{-2}\partial_ur\partial_vr)\bigg)\nonumber\\
&=\frac{\partial_ur}{2}+2\partial_ur\Omega^{-2}\partial_ur\partial_vr+2r\partial_u(\Omega^{-2}\partial_ur)\partial_vr+2r\Omega^{-2}\partial_ur\partial_u\partial_vr\nonumber\\
&=\frac{\partial_ur}{2}+2\partial_ur\Omega^{-2}\partial_ur\partial_vr-8\pi r^2\Omega^{-2}\partial_vr|\partial_u\phi|^2\nonumber\\
&\hspace{0.5cm}+2\Omega^{-2}\partial_ur\bigg(-\frac{\Omega^2}{4}\big(1-\frac{Q^2}{r^2}\big)-\partial_vr\partial_ur\bigg)\nonumber\\
&=-8\pi r^2\Omega^{-2}\partial_vr|D_u\phi|^2+\frac{Q^2\partial_ur}{2r^2}\label{massu},\\
\partial_vm &= \partial_v\bigg(\frac{r}{2}(1+4\Omega^{-2}\partial_ur\partial_vr)\bigg)\nonumber\\
&=\frac{\partial_vr}{2}+2\partial_vr\Omega^{-2}\partial_ur\partial_vr+2r\partial_v(\Omega^{-2}\partial_vr)\partial_ur+2r\Omega^{-2}\partial_vr\partial_u\partial_vr\nonumber\\
&=\frac{\partial_ur}{2}+2\partial_ur\Omega^{-2}\partial_ur\partial_vr-8\pi r^2\Omega^{-2}\partial_ur|\partial_v\phi|^2\nonumber\\
&\hspace{0.5cm}+2\Omega^{-2}\partial_vr\bigg(-\frac{\Omega^2}{4}\big(1-\frac{Q^2}{r^2}\big)-\partial_vr\partial_ur\bigg)\nonumber\\
&=-8\pi r^2\Omega^{-2}\partial_ur|\partial_v\phi|^2+\frac{Q^2\partial_vr}{2r^2}\label{massv}.
\end{align}

Finally, we define what is meant by a trapped surface.
\begin{definition}
	A \textit{trapped surface} $S$ in a spherically symmetric spacetime is a point $(u,v)$ on the Penrose diagram (which represents a sphere) such that $\partial_ur(u,v)<0$ and $\partial_vr(u,v)<0$. If $\partial_vr(u,v) = 0$, we call $(u,v)$ a \textit{marginally outer trapped surface (MOTS)}.
\end{definition}

In {\color{black}the following}, we will only focus on a narrow strip of the double null coordinate patch $[u_0,0]\times [v_1,v_2]$, for some $v_2>v_1$. We are going to give conditions under which trapped surface formation is guaranteed in this strip. We {\color{black}introduce:} 

\begin{gather}
r_i(u) := r(u,v_i), \hspace{0.5cm} m_i(u):=m(u,v_i),\hspace{0.5cm} i = 1,2\nonumber\\
\delta(u):=\frac{r_2(u)}{r_1(u)}-1,\hspace{0.5cm} \delta_0 := \delta(u_0)\nonumber\\
\eta(u):=\frac{2(m_2(u)-m_1(u))}{r_2(u)},\hspace{0.5cm} \eta_0 := \eta(u_0)\nonumber\\
x(u):=\frac{r_2(u)}{r_2(u_0)}\label{quantities}
\end{gather}
See Figure \ref{doublenull2} for an illustration. Henceforth, any dynamical quantity (except for $u$ and $v$) with the subscript $\{1,2\}$ shall be treated as a function of $u$ with $v = v_i, i = \{1,2\}$ fixed. Furthermore, denote {\color{black} the region $[u_0,0]\times[v_1,v_2]$ by $\mathcal{R}$.}\\
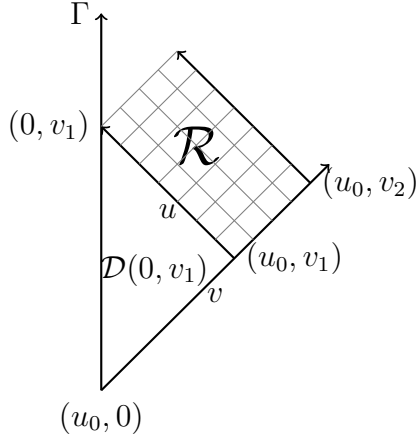
\begin{figure}\label{doublenull2}
	\centering
	\begin{tikzpicture}
	\node[] at (1.25,3.25) {\LARGE $\mathcal{R}$};
	\node[] at (0.7,1.6) { $\mathcal{D}(0,v_1)$};
	\begin{scope}[thick]
	\draw[->] (0,0) node[anchor=north]{$(u_0,0)$} -- (0,5)node[anchor = east]{$\Gamma$};
	\draw[->] (0,0) --node[anchor = north]{$v$} (3,3);
	\draw[->] (1.75,1.75) node[anchor=west]{$(u_0,v_1)$} --node[anchor=north]{$u$} (0,3.5)node[anchor = east]{$(0,v_1)$};
	\draw[->] (2.75,2.75) node[anchor=west]{$(u_0,v_2)$} -- (1,4.5);
	\end{scope}
	\begin{scope}[gray]
	\draw (2,2) -- (0.25,3.75);
	\draw (2.25,2.25) -- (0.5,4);
	\draw (2.5,2.5) -- (0.75,4.25);
	\draw(1.5,2) -- (2.5,3);
	\draw(1.25,2.25) -- (2.25,3.25);
	\draw(1,2.5) -- (2,3.5);
	\draw(0.75,2.75) -- (1.75,3.75);
	\draw(0.5,3) -- (1.5,4);
	\draw(0.25,3.25) -- (1.25,4.25);
	\draw(0,3.5) -- (1,4.5);
	\end{scope}
	\end{tikzpicture}
	
	\caption{Problem setup on Penrose Diagram}
	\label{doublenull2}
\end{figure}

\section{A Trapped Surface formation criterion for the Complex Scalar Field}\label{main section}

\subsection{Outline} {\color{black}Before giving the complete proof, we briefly describe the main ideas. 
\begin{enumerate}
\item First, we prove that $r(u,v)$ is decreasing with respect to $u$ in Lemma \ref{negativeincomingcharged}, hence the dimensionless length scale $x(u):=\frac{r_2(u)}{r_2(u_0)}$ decreases as $u$ increases, and $x(u_0) = 1$. 

\item Then we employ a proof-by-contradiction argument: assuming that $\mathcal{D}(0,v_1)\cup\mathcal{R}$ does not have a trapped surface, we derive an inequality for $\eta$ in terms of $x$ in the region $[u_0,u_*]\times[v_1,v_2]$. We further show that $\frac{d\eta}{dx}$ is bounded from above, i.e. $\frac{d\eta}{du}$ is bounded from below, and therefore we get a lower bound on $\eta(u_*)$. 

\noindent If this lower bound is greater than $1$, i.e., $\eta(u_*)= \frac{2(m_2-m_1)}{r_2}(u_*)>1$, it implies $\frac{2m_2}{r_2}(u_*)>1$ and this means $S(u_*,v_2)$ \text{ is a trapped surface.} Since $(u_*,v_2)$ is a point in $\mathcal{R}$, the above gives us the desired contradiction. 
\end{enumerate}

The key of above arguments is to bound $\frac{d\eta}{du}$. A direct computation gives:
\begin{align}
\frac{d\eta}{dx}&=-\frac{\eta}{x}-\frac{16\pi\partial_vr_2\Omega_2^{-2}}{x\partial_ur_2}\bigg(r_2^2|D_u\phi_2|^2-\frac{\Omega_1^{-2}\partial_vr_1}{\Omega_2^{-2}\partial_vr_2}r_1^2|D_u\phi_1|^2\bigg)+\frac{Q_2^2}{xr_2^2}.\label{1.10intro}
\end{align}
We show  in Lemma \ref{bound for charge} that the \underline{non}-supercharged assumption along $v=v_1$ allows us to show that ${Q_2^2}/{r_2^2}$ remains bounded by $\eta$, plus a small error term. If $\eta$ is large enough compared to ${Q_2^2}/{r_2^2}$, then the error can be absorbed into $\eta$.

\noindent With the control of ${Q_2^2}/{r_2^2}$ in terms of $\eta$, we further have
\begin{gather}
	\Theta^2:=\big(r_2|D_u\phi_2|-r_1|D_u\phi_1\big)^2\lesssim \frac{-\partial_ur_2}{8\pi\Omega_2^{-2}\partial_vr_2}(m_2-m_1)\bigg(\frac{1}{r_1}-\frac{1}{r_2}\bigg)\label{keylemma1intro}\\
	\text{and }
 \frac{\Omega_2^{-2}\partial_vr_2}{\Omega_1^{-2}\partial_vr_1}(u)\lesssim e^{-\eta(u)}\label{keylemma2intro}.
\end{gather}
We can substitute $\eqref{keylemma1intro}$ and $\eqref{keylemma2intro}$ into $\eqref{1.10intro}$, to obtain
\begin{align}
\frac{d\eta}{dx}\lesssim -\frac{\eta}{x}\bigg(1-\frac{\delta_0}{x(1+\delta_0)-\delta_0}\bigg)+\frac{1}{x}\frac{\delta_0}{x(1+\delta_0)-\delta_0}.\label{diffineq}
\end{align}
Integrating this will give us a lower bound on $\eta(u)$ for $u\in[u_0,u_*]$. This ensures that for $u\in[u_0,u'], \eta(u)$ is always large enough compared to $\frac{Q_2^2}{r_2^2}$, so that the differential inequality is always valid in a neighbourhood of $u'$, and hence the domain of validity of $\eqref{diffineq}$ can be extended to the whole $[u_0,u_*]$. Finally, the inequality also shows that $\eta(u_*)>1$, a contradiction to the no-trapped-surfaces assumption. Hence the initial assumption that $\mathcal{D}(0,v_1)\cup\mathcal{R}$ has no trapped surfaces cannot be true and this completes the argument.\\
}

\subsection{Proof of Theorem \ref{thm1.3}}
To begin the proof proper, we first give a (negative) upper bound for $\partial_ur$ in the region $[u_0,0]\times[v_1,\infty)$ as promised.  The following lemma is the analog of Lemma \ref{negativeincoming} in the uncharged case. This has two important consequences described in the remarks.\\
\begin{lemma}\label{negativeincomingcharged}
	{\color{black}$\partial_ur\leq -\frac{1-\epsilon}{2}\Omega^2$ everywhere in $\mathcal{D}(0,v_1)\cup\big([u_0,0]\times[v_1,\infty)\big)$.}
	\begin{proof}
		Rewrite $\eqref{EMS1}$ as
		\begin{align*}
		\partial_v\big(r\partial_ur\big)=-\frac{\Omega^2}{4}\bigg(1-\frac{Q^2}{r^2}\bigg).
		\end{align*}Applying the assumption that $\epsilon < 1$ {\color{black}and $\Omega^2=1$} on $C$, the following inequalities hold on $C$:
		\begin{align*}
		-\frac{1}{4}\leq\partial_v(r\partial_ur)(u_0,v)\leq-\frac{1}{4}+\frac{\epsilon}{4}.
		\end{align*}
		Integrating both sides and dividing by $r$:
		\begin{align}
		-\frac{v}{4r(u_0,v)}\leq \partial_ur(u_0,v)\leq -\frac{v(1-\epsilon)}{4r(u_0,v)}\label{estimate for nu}.
		\end{align}
		{\color{black}For} the first inequality of $\eqref{estimate for nu}$  at $v=0$, we get
		\begin{align}\label{boundoutgoing}
		-\frac{1}{4\partial_vr(u_0,0)}\leq\partial_ur(u_0,0) = -\partial_vr(u_0,0)\implies\partial_vr(u_0,0)\leq\frac{1}{2}.
		\end{align}
		Since $\Omega^2 = 1$ on $C$, $\eqref{EMS4}$ gives us $\partial_v\partial_vr\leq 0$, i.e. $r$ is concave with respect to {\color{black}$v$}. Combining this with the fact that $r(u_0,0) = v(u_0,0) = 0$, we have:
		\begin{align*}
		\frac{r}{v}(u_0,v)\leq\partial_vr(u_0,0).
		\end{align*}
		Substituting this into the second inequality of $\eqref{estimate for nu}$, followed by applying $\eqref{boundoutgoing}$, we get
		\begin{align*}
		\partial_ur(u_0,v)\leq - \frac{1-\epsilon}{4\partial_vr(u_0,0)}\leq -\frac{1-\epsilon}{2}.
		\end{align*}
		By $\eqref{EMS3}$, $\Omega^{-2}\partial_ur$ is decreasing along incoming null geodesics. Hence for a general point in {\color{black} $\mathcal{D}(0,v_1)\cup\big([u_0,0]\times[v_1,\infty)\big)$}, we have $\Omega^{-2}\partial_ur\leq-\frac{1-\epsilon}{2}$.
	\end{proof}
\end{lemma} 

\begin{remark}
	Under the assumption of no trapped surfaces, $m(u,v)\geq 0$ for all {\color{black}$(u,v)\in \mathcal{D}(0,v_1)\cup\big([u_0,0]\times[v_1,\infty)\big)$}
\end{remark}
\begin{proof} 
	
	Given any point $(u,v)\in \mathcal{R}$, we can extend the outgoing null geodesic backwards until it intersects $\Gamma$ at some coordinate $(u,v_c)$, so that $r(u,v_c) = 0$. Using $\eqref{massv}$, we have
	\begin{align*}
	\partial_vm = -8\pi r^2\Omega^{-2}\partial_ur|\partial_v\phi|^2 + \frac{Q^2\partial_vr}{2r^2}.
	\end{align*}
	
	{\color{black}Since} $\partial_ur\leq 0$ by Lemma $\ref{negativeincomingcharged}$, and $\partial_vr> 0$ by the no trapped surface assumption, we get $\partial_vm\geq 0$. Combining with the fact that $m(u,v_c) = 0$, we obtain the desired result.
\end{proof}
\subsection{Estimates for $Q,r$}\label{sec:nothing}
In this section we will bound $Q$ in terms of the Hawking mass in $\mathcal{R}${\color{black}, and show that $\partial_u\partial_vr \leq 0$} under appropriate conditions. Obtaining a bound on $Q$ will require a bound on $\frac{r_2}{r_1}$, which in turn requires a bound on $Q$. Hence we will develop these bounds using a bootstrap argument.
\begin{proposition}\label{A priori estimate on Q} {\color{black}Fix $0<\o<\f23$. Choose $v_2-v_1$ sufficiently small satisfying \eqref{first assumption on v2-v1} and \eqref{second assumption on v2-v1}.} Let $v_1<v_a\leq v_2$. Assume that $\frac{r_2(u)}{r_1(u)}\leq\frac{3}{2}$, {\color{black}i.e., $\delta(u)\leq \frac12$ } for $u\in[u_0,0]$, and that {\color{black}$\mathcal{R}:=[u_0,0]\times[v_1,v_2]$} is free of trapped surfaces. Then the following inequality holds:
	\begin{align*}
	\frac{Q_a^2(u)}{r_a^2(u)}\leq\frac{\omega}{4}\eta_a(u)+{\color{black}\f{2Q_1^2(u)}{r^2_1(u)}},
	\end{align*}
where 
\begin{align*}
    {\color{black}\eta_a(u):=\frac{2\big(m_a(u)-m_1(u)\big)}{r_a}}.
\end{align*}
Over here the subscript $a$ indicates a quantity evaluated at the point  $(u,v_a)$.
\end{proposition}
\begin{proof}
{\color{black}We write $Q_a$ as the integral of its derivative:}
\begin{align}\label{chargeestimate1}
Q_a^2 &=\bigg(\int_{v_1}^{v_a}\partial_vQ\text{ }dv +Q_1\bigg)^2\leq 2\bigg(\int_{v_1}^{v_a}\partial_vQ\text{ }dv\bigg)^2+2Q_1^2\nonumber\\
&\leq 2\bigg(\int_{v_1}^{v_a}4\pi\mathfrak{e}|\phi||\partial_v\phi|r^2dv\bigg)^2+2Q_1^2,\text{ by applying \eqref{EMS9}}\nonumber\\
&\leq 32\pi^2\mathfrak{e}^2\int_{v_1}^{v_a}r^2|\phi|^2dv\cdot\int_{v_1}^{v_a}r^2|\partial_v\phi|^2dv+2Q_1^2.
\end{align}
{\color{black}The second integral in the previous line can be bounded:}
\begin{align}\label{chargeestimate2}
\int_{v_1}^{v_a}r^2|\partial_v\phi|^2dv=\int_{v_1}^{v_a}\frac{-8\pi r^2\Omega^{-2}\partial_ur|\partial_v\phi|^2}{-8\pi\Omega^{-2}\partial_ur}dv\nonumber\\
\leq\frac{1}{4\pi(1-\epsilon)}\int_{v_1}^{v_a}\partial_vm-\frac{Q^2\partial_vr}{2r^2}\text{ }dv\leq\frac{m_a-m_1}{4\pi(1-\epsilon)},
\end{align}
where we applied Lemma \ref{negativeincomingcharged} in the second last inequality to pull out the term {\color{black}$\Omega^{-2}\partial_ur$} in the denominator{\color{black}, and} used the assumption that $\partial_vr\geq 0$ for the last inequality. Next, we bound the first integral:\textcolor{black}{
\begin{align}\label{chargeestimate3}
\int_{v_1}^{v_a}r^2|\phi|^2dv
&=\int_{v_1}^{v_a}\bigg(r^2\bigg|\int_{v_1}^v\partial_{v'}\phi \text{ }dv'+\phi_1\bigg|^2\bigg)dv\nonumber\\
&\leq 2r_a^2\int_{v_1}^{v_a}\bigg[\bigg(\int_{v_1}^v\partial_{v'}\phi\text{ } dv'\bigg)^2+|\phi_1|^2\bigg]dv\nonumber\\
&\leq 2r_a^2\int_{v_1}^{v_a}\bigg[(v-v_1)\int_{v_1}^v|\partial_{v'}\phi|^2dv'+|\phi_1|^2 \bigg]dv\nonumber\\
&\leq 2r_a^2\int_{v_1}^{v_a}\bigg[(v-v_1)\int_{v_1}^v\bigg(\frac{-8\pi r^2\Omega^{-2}\partial_ur|\partial_v\phi|^2}{-8\pi r^2\Omega^{-2}\partial_ur}\bigg)dv'+|\phi_1|^2 \bigg]dv\nonumber\\
&\leq\frac{2r_a^2(v_a-v_1)}{r_1^2}\frac{1}{8\pi}\frac{2}{1-\epsilon}\int_{v_1}^{v_a}\int_{v_1}^v\partial_{v'}m\text{ }dv'dv+2r_a^2\int_{v_1}^{v_a}|\phi_1|^2dv,\nonumber \\
&\hspace{0.5cm}\text{ by Lemma }\ref{negativeincomingcharged}\nonumber\\
&\leq\frac{v_a-v_1}{2\pi(1-\epsilon)}\bigg(\frac{r_a}{r_1}\bigg)^2\int_{v_1}^{v_a}(m_a-m_1)dv+2r_a^2(v_a-v_1)|\phi_1|^2\nonumber\\
&\leq\frac{v_a-v_1}{2\pi(1-\epsilon)}\bigg(\frac{r_a}{r_1}\bigg)^2(v_a-v_1)(m_a-m_1)+2r_a^2(v_a-v_1)|\phi_1|^2.
\end{align}}
Substituting $\eqref{chargeestimate2}$ and $\eqref{chargeestimate3}$ back into $\eqref{chargeestimate1}$ and rearranging, we get
\begin{align*}
Q_a^2\leq \frac{4\mathfrak{e}^2(v_a-v_1)^2}{(1-\epsilon)^2}\bigg(\frac{r_a}{r_1}\bigg)^2(m_a-m_1)^2+\frac{16\pi\mathfrak{e}^2}{1-\epsilon}r_a^2(m_a-m_1)(v_a-v_1)|\phi_1|^2+2Q_1^2.
\end{align*}
Dividing both sides by $r_a^2$, we get\textcolor{black}{
\begin{align*}
\frac{Q_a^2}{r_a^2}&\leq\frac{\mathfrak{e}^2(v_a-v_1)^2}{(1-\epsilon)^2}\bigg(\frac{r_a}{r_1}\bigg)^2\eta_a^2+\frac{8\pi\mathfrak{e}^2(v_a-v_1)}{1-\epsilon}\bigg(\frac{r_a}{r_1}\bigg)(r_1|\phi_1|^2)\eta_a+2\frac{Q_1^2}{r_a^2}\\
&\leq\bigg(\frac{9\mathfrak{e}^2}{4(1-\epsilon)^2}(v_a-v_1)^2\eta_a+\frac{12\pi L\mathfrak{e}^2}{1-\epsilon}(v_a-v_1)\bigg)\eta_a+2\frac{Q_1^2}{r_1^2}, \hspace{0.5cm}\text{since }\frac{r_a}{r_1}\leq\frac{3}{2}\\
&\leq\bigg(\frac{9\mathfrak{e}^2}{4(1-\epsilon)^2}(v_a-v_1)^2+\frac{12\pi L\mathfrak{e}^2}{1-\epsilon}(v_a-v_1)\bigg)\eta_a+2{\color{black}\f{Q_1^2}{r_1^2}},
\end{align*}}
where in the last inequality we have used the fact that $\eta_a<1$. This is because the no trapped surface or MOTS assumption gives:
  \begin{align*}
     \eta_a\leq\frac{2(m_a-m_1)}{r_a}\leq\frac{2m_a}{r_a}< 1.
 \end{align*}
Hence by the assumption \eqref{first assumption on v2-v1} on $v_2-v_1$, we have $\frac{Q_a^2}{r_a^2}\leq \frac{\omega}{4}\eta_a+2\epsilon$.
\end{proof}
We wish to get rid of the $\epsilon$ term in the upper bound given by the previous proposition. This will be done in Lemma $\ref{bound for charge}$. For that, we will need the next proposition which is the equivalent of Proposition \ref{mixed derivatives of r real} in the uncharged case. {\color{black}This proposition is proven} using a bootstrap argument.
\begin{proposition}\label{mixed derivatives of r} 
	Assume {\color{black}the initial data along $\underline{C}$ is not super-charged}, and $\mathcal{R}$ is free of trapped surfaces. Then $\partial_u\partial_vr \leq 0$ in {\color{black}$[u_0,u_*]\times[v_1,v_2]$} and {\color{black} $\delta(u):=\frac{r_2}{r_1}-1\leq\frac{1}{2}$} for $u\in[u_0,u_*]${\color{black}, where $u_*$ is defined such that $x(u_*)=\frac{3\delta_0}{1+\delta_0}$.}
\end{proposition}
\begin{proof}
	Let $x':=\inf\big\{x\in[\frac{3\delta_0}{1+\delta_0},1]\big|\delta(y)\leq\frac{1}{2}\text{ holds for } y\in[x,1]\big\}$. We will aim to show that $x'=\frac{3\delta_0}{1+\delta_0}$. {\color{black}This proves the claim that $\delta(u)\leq\frac{1}{2}$ for $u\in[u_0,u_8]$, as $x$ is monotonically decreasing with respect to $u$.}\\
	
	Since we have $\delta(x')\leq\frac{1}{2}$, it follows that $\frac{r_2(x')}{r_1(x')}\leq\frac{3}{2}$ and hence we can apply Proposition \ref{A priori estimate on Q}. Thus, for every {\color{black}$x\in [x',1],v_a\in [v_1,v_2]$}, we have
	{\color{black}
\begin{align*}
\frac{Q^2}{r^2}(x,v_a)&\leq\frac{\omega}{4}\eta_a+\frac{2Q_1^2}{r^2}\leq\eta_a+\frac{2Q_1^2}{r^2}\\
&=\frac{2m_a}{r}-\frac{2m_1}{r}+\frac{2Q_1^2}{r^2}=\frac{2m_a}{r}-\frac{2}{r}\bigg(m_1-\frac{Q_1^2}{r}\bigg)\\
&=\frac{2m_a}{r}-\frac{2}{r}\bigg(m_1-\frac{Q_1^2}{r_1}\bigg).
\end{align*}}
\textcolor{black}{By the non-supercharged assumption, \eqref{m Q 2} from Remark $\ref{nonsupercharged}$ tells us that $m_1-\frac{Q_1^2}{r_1}\geq 0$}. Hence we have
\begin{align*}
\frac{Q^2}{r^2}(x,v_a)\leq \frac{2m}{r}(x,v_a).
\end{align*}Now we rewrite \eqref{EMS1} into the following equivalent form:
\begin{align}\label{EMS1mass}{\color{black}
\partial_u\partial_vr=-\frac{\Omega^2}{4r}\bigg(\frac{2m}{r}-\frac{Q^2}{r}\bigg)}.
\end{align}
Since we have just shown that {\color{black}$\frac{2m}{r}\geq\frac{Q^2}{r}$} for $x\in [x',1]$, it follows that $\partial_u\partial_vr\leq 0$ in the region $[u_0,u']\times[v_1,v_2]$.\\

Integrating with respect to $u$, we get:  

\begin{align*}
\partial_vr(u)-\partial_vr(u_0)\leq 0\implies\partial_vr(u)\leq\partial_vr(u_0).
\end{align*}
Integrating the last inequality above with respect to $v$, we obtain
\begin{align*}
r_2(u)-r_1(u)\leq r_2(u_0)-r_1(u_0), \hspace{0.5cm}\text{for all } u\in [u_0,u'].
\end{align*}
{\color{black}Here $u'$ is defined so that $x(u')=x'$}. We can use {\color{black}this} to derive a bound for $\delta(u)$:
\begin{align*}
\delta(u) = \frac{r_2}{r_1}-1&=\frac{r_2-r_1}{r_2-(r_2-r_1)}\leq\frac{r_2(u_0)-r_1(u_0)}{r_2(u)-(r_2(u_0)-r_1(u_0))}\\
&\leq\frac{\delta_0}{\frac{r_2(u)}{r_1(u_0)}-\delta_0}=\frac{\delta_0}{x(u)(1+\delta_0)-\delta_0},\hspace{0.5cm}\text{ for all }u\in[u_0,u'].
\end{align*}
Hence{\color{black},} if $x'>\frac{3\delta_0}{1+\delta_0}$, we have $\delta(x')<\frac{\delta_0}{3\delta_0-\delta_0}=\frac{1}{2}$. By the continuity of the function $\delta(x)$, there exist some $x''<x'$ such that $\delta(x)<\frac{1}{2}$ for all $x\in[x'',1]$, which is a contradiction to infimum property of $x'$. {\color{black}Therefore, we must have} $x' = \frac{3\delta_0}{1+\delta_0}$.
\end{proof}
{\color{black}
\begin{remark}
	Since all subsequent Lemmas and Propositions depend on Proposition $\ref{A priori estimate on Q}$, the non-supercharged hypothesis is necessary for all of them.
\end{remark}
}

{\color{black}Recall {\color{black} that} $\epsilon:=\sup_{\underline{C}}\f{Q^2}{r^2}<1$}. Combining Proposition \ref{A priori estimate on Q} and Proposition \ref{mixed derivatives of r}, we get the following lemma:
\begin{lemma}\label{bound for charge}
	Assume that {\color{black}the initial data along $\underline{C}$ is not super-charged} and that $\mathcal{R}$ is free of trapped surfaces. Then for every {\color{black}$(u,v_a)\in[u_0,u_*]\times[v_1,v_2]$}, we have the following estimate for the charge:
	\begin{align*}
	{\color{black}\frac{Q_a^2}{r_a^2}}\leq\frac{\omega}{4}{\color{black}\eta_a}+2\epsilon.
	\end{align*}
	Furthermore, if {\color{black}$\eta_a\geq\frac{8\epsilon}{\omega}$, then $\frac{Q_a^2}{r_a^2}\leq\frac{\omega}{2}\eta_a$.}
\end{lemma}
\begin{proof}
	The first part of the lemma is almost proven: Since the hypothesis of this lemma satisfies that of Proposition $\ref{mixed derivatives of r}$, we have $\delta(u)\leq\frac{3}{2}$ for all $u\in [u_0,u_*]$, which is the hypothesis of Proposition $\ref{A priori estimate on Q}${\color{black}. This gives} us the first part of the Lemma. The second part follows from a computation. {\color{black}$\eta_a\geq\frac{8\epsilon}{\omega}$} implies that {\color{black}$2\epsilon\leq \f{\o}{4}\eta_a.$ Hence,
$$\f{Q_a^2}{r_a^2}\leq \f{\o}{4}\eta_a+2\e\leq \f{\o}{4}\eta_a+\f{\o}{4}\eta_a=\f{\o}{2}\eta_a.$$
	
}

\end{proof}
\subsection{Estimates for $D_u\phi$, $\partial_vr$}
In this section, we will prove two lemmas which hold in the region {\color{black}$[u_0,u_*]\times[v_1,v_2]$ under} the premise that the region is free of trapped surfaces. These are the {\color{black}equivalents of Lemmas} \ref{keylemma1real} and \ref{keylemma2real} in the uncharged case.
\begin{lemma}\label{keylemma1}
	Define $\Theta:=r_2|D_u\phi_2|-r_1|D_u\phi_1|$. Suppose that {\color{black}the initial data along $\underline{C}$ is not super-charged} and $\mathcal{D}(0,v_1)\cup\mathcal{R}$ is free of trapped surfaces. {\color{black}If $\eta\geq \f{8\epsilon}{\o}$, then }
	\begin{align*}
	\Theta(u)^2\leq\bigg(1+\frac{\omega}{2}\bigg) \frac{-\partial_ur_2}{8\pi\Omega_2^{-2}\partial_vr_2}(m_2-m_1)\bigg(\frac{1}{r_1}-\frac{1}{r_2}\bigg)(u)
	\end{align*}
	for all $u\in [u_0,u_*]$.
\end{lemma}
\begin{proof}
	By integrating equation $\eqref{complex wave equation}$, we get
	\begin{align}
	\Theta^2 &= \big(r_2|D_u\phi_2|-r_1|D_u\phi_1|\big)^2\nonumber\\
	&\leq|r_2D_u\phi_2-r_1D_u\phi_1|^2=\bigg|\int_{v_1}^{v_2}-\partial_ur\partial_v\phi -i\mathfrak{e}\frac{Q\phi\Omega^2}{4r}dv\bigg|^2\nonumber\\
	&\leq(1+\kappa)\bigg|\int_{v_1}^{v_2}-\partial_ur|\partial_v\phi| dv\bigg|^2+\big(1+\frac{1}{\kappa}\big)\mathfrak{e}^2\bigg|\int_{v_1}^{v_2}\frac{Q\phi\Omega^2}{4r}dv\bigg|^2\text{, for any }\kappa>0\nonumber\\
	&\leq\frac{1+\kappa}{8\pi}\int_{v_1}^{v_2}-8\pi r^2\partial_ur\Omega^{-2}|\partial_v\phi|^2dv\int_{v_1}^{v_2}-\frac{\partial_ur}{r^2\Omega^{-2}}dv+\big(1+\frac{1}{\kappa}\big)\mathfrak{e}^2\bigg|\int_{v_1}^{v_2}\frac{Q\phi\Omega^2}{4r}dv\bigg|^2,\label{lemma1eqn1}
	\end{align}
	{\color{black}where we used} Holder's inequality for the last inequality.\\
	
	{\color{black}We bound the first summand like how we did in} Lemma \ref{keylemma1real}. We bound the first integral of the first {\color{black}term}:
	\begin{align}
	\int_{v_1}^{v_2}-8\pi r^2\partial_ur\Omega^{-2}|\partial_v\phi|^2dv&=\int_{v_1}^{v_2}\partial_vm-\frac{Q^2\partial_vr}{2r^2}dv\nonumber\\
	&\leq\int_{v_1}^{v_2}\partial_vm\text{ }dv, \text{ since }\partial_vr> 0\nonumber\\
	&=m_2-m_1\label{lemma1eqn2}.
	\end{align}
	To bound the second integral of the first {\color{black}term}, we apply Proposition $\ref{mixed derivatives of r}$ to get $\partial_v\partial_ur\leq 0$, and hence $\partial_ur\geq\partial_ur_2$. Also, equation $\eqref{EMS4}$ implies that $\Omega_2^{-2}\partial_vr_2\leq\Omega^{-2}\partial_vr$. Combining these two pieces of information, we have
	
	\begin{align}
	\int_{v_1}^{v_2}-\frac{\partial_ur}{r^2\Omega^{-2}}dv&={\color{black}\int_{r_1}^{r_2}-\frac{\partial_ur}{r^2\Omega^{-2}\partial_vr}dr}\nonumber\\
	&{\color{black}\leq-\partial_ur_2\int_{r_1}^{r_2}\frac{1}{r^2\Omega^{-2}\partial_vr}dr}\nonumber\\
	&\leq\frac{-\partial_ur_2}{\Omega_2^{-2}\partial_vr_2}\int_{r_1}^{r_2}\frac{1}{r^2}dr\nonumber\\
	&=\frac{\partial_ur_2}{\Omega_2^{-2}\partial_vr_2}\big(\frac{1}{r_2}-\frac{1}{r_1}\big)\label{lemma1eqn3}.
	\end{align}
Substituting $\eqref{lemma1eqn2}$ and $\eqref{lemma1eqn3}$ back into $\eqref{lemma1eqn1}$ gives us:
\begin{align}
	\Theta(u)^2\leq (1+\kappa)\frac{\partial_ur_2}{8\pi\Omega_2^{-2}\partial_vr_2}(m_2-m_1)\bigg(\frac{1}{r_2}-\frac{1}{r_1}\bigg)+\big(1+\frac{1}{\kappa}\big)\mathfrak{e}^2\bigg|\int_{v_1}^{v_2}\frac{Q\phi\Omega^2}{4r}dv\bigg|^2\label{Thetabound}.
\end{align}		
To bound the remaining integral, {\color{black}we apply the Cauchy-Schwarz inequality and  Lemma \ref{bound for charge}:}
\begin{align}
\bigg|\int_{v_1}^{v_2}\frac{Q\phi\Omega^2}{4r}dv\bigg|^2&= \bigg|\frac{1}{4}\int_{v_1}^{v_2}\frac{Q}{r^\frac{3}{2}\Omega^{-2}}r^\frac{1}{2}\phi\text{ }dv\bigg|^2\nonumber\\
&\leq\frac{1}{16}\int_{v_1}^{v_2}\frac{Q^2}{r^2}\frac{1}{r}\frac{1}{\Omega^{-2}}\frac{1}{\Omega^{-2}}dv\cdot\int_{v_1}^{v_2}r|\phi|^2dv\nonumber\\
&\leq{\color{black}\frac{1}{16}\int_{r_1}^{r_2}\frac{\omega}{2}\eta_a\frac{1}{r}\frac{1}{\Omega^{-2}\partial_vr}\frac{-\partial_ur}{-\Omega^{-2}\partial_ur}dr\int_{v_1}^{v_2}r|\phi|^2dv}\label{lastterm}.
\end{align} 
We bound the first integral: {\color{black} by Proposition \ref{mixed derivatives of r} we have $\partial_u\partial_v r\leq 0$, which implies $\partial_ur_2\leq\partial_ur$.} And with $\Omega_2^{-2}\partial_vr_2\leq\Omega^{-2}\partial_vr$ by \eqref{EMS4}, we derive
\begin{align}
\int_{r_1}^{r_2}\frac{\omega}{2}\eta_a\frac{1}{r}\frac{1}{\Omega^{-2}\partial_vr}\frac{-\partial_ur}{-\Omega^{-2}\partial_ur}dr&\leq\frac{-\partial_ur_2}{\Omega_2^{-2}\partial_vr_2}\int_{r_1}^{r_2}\frac{\omega}{2}\frac{2(m_a-m_1)}{r_a}\frac{1}{r}\frac{1}{-\Omega^{-2}\partial_ur}dr\nonumber\\
&\leq \frac{-2\omega}{1-\epsilon}\frac{\partial_ur_2}{\Omega_2^{-2}\partial_vr_2}\int_{r_1}^{r_2}(m_2-m_1)\frac{1}{r^2}dr, \nonumber\\&\hspace{0.5cm}{\color{black}\text{by Lemma \ref{negativeincomingcharged}, (} \Omega^{-2}\partial_u r\leq-\frac{1-\epsilon}{2}\text{)}}\nonumber\\
&=\frac{-2\omega}{1-\epsilon}\frac{\partial_ur_2}{\Omega_2^{-2}\partial_vr_2}(m_2-m_1)\bigg(\frac{1}{r_1}-\frac{1}{r_2}\bigg)\nonumber\\
&=-\frac{2\omega}{1-\epsilon}\frac{\partial_ur_2}{\Omega_2^{-2}\partial_vr_2}(m_2-m_1)\bigg(\frac{1}{r_1}-\frac{1}{r_2}\bigg)\label{firstintegral}.
\end{align}
Now we bound the second integral: 
\begin{align*}
&\int_{v_1}^{v_2}r|\phi|^2dv= \int_{v_1}^{v_2}r\bigg|\int_{v_1}^{v'}\partial_v\phi\text{ }dv +\phi_1\bigg|^2dv'\\
\leq& 2\int_{v_1}^{v_2}\bigg(r\bigg|\int_{v_1}^{v'}\partial_v
\phi\text{ }dv\bigg|^2+r|\phi_1|^2\bigg)dv'\\
\leq& 2\int_{v_1}^{v_2}\bigg(r(v'-v_1)\int_{v_1}^{v'}|\partial_v
\phi|^2\text{ }dv+r|\phi_1|^2\bigg)dv'\\
\leq& 2r_2\int_{v_1}^{v_2}\bigg((v_2-v_1)\int_{v_1}^{v'}\frac{-8\pi\Omega^{-2}r^2\partial_ur|\partial_v\phi|^2}{-8\pi\Omega^{-2}\partial_urr^2}dv+|\phi_1|^2\bigg)dv'\\
\leq& 2r_2\int_{v_1}^{v_2}\bigg(\frac{2(v_2-v_1)}{(1-\epsilon)8\pi r_1^2}\int_{v_1}^{v'}\partial_vm\text{ }dv+|\phi_1|^2\bigg)dv'\\
=&2r_2\int_{v_1}^{v_2}\bigg(\frac{v_2-v_1}{(1-\epsilon)4\pi r_1^2}(m_2-m_1)+|\phi_1|^2\bigg)dv'\\
\leq& \frac{r_2^2}{r_1^2}\frac{(v_2-v_1)^2}{4\pi(1-\epsilon)}\frac{2(m_2-m_1)}{r_2}+2r_2(v_2-v_1)|\phi_1|^2.
\end{align*}
Using the no trapped surface {\color{black}or MOTS assumption, we have} $\frac{2(m_2-m_1)}{r_2}=
\eta<1$. Using Lemma \ref{mixed derivatives of r}, $\frac{r_2}{r_1}\leq\frac{3}{2}$. Hence we have:
\begin{align}
\int_{v_1}^{v_2}r|\phi|^2dv\leq\frac{9(v_2-v_1)^2}{16\pi(1-\epsilon)}+2r_2(u_0)(v_2-v_1)|\phi_1|^2\leq\omega\frac{1-\epsilon}{320\mathfrak{e}^2\pi}, \label{secondintegral}
\end{align}
where we use the assumption in \eqref{second assumption on v2-v1}.

Substituting \eqref{firstintegral} and \eqref{secondintegral} back into \eqref{lastterm}, we get
{\color{black}\begin{align*}
\bigg|\int_{v_1}^{v_2}\frac{Q\phi\Omega^2}{4r}dv\bigg|^2\leq-\frac{\omega^2}{160\mathfrak{e}^2\pi}\frac{\partial_ur_2}{\Omega_2^{-2}\partial_vr_2}(m_2-m_1)\bigg(\frac{1}{r_1}-\frac{1}{r_2}\bigg).
\end{align*}}
Now, set $\kappa = \frac{\omega}{4}$ in \eqref{Thetabound} and utilize the inequality above:
\begin{align*}
\Theta^2 \leq&\big(1+\frac{\omega}{4}\big)\frac{-\partial_ur_2}{8\pi\Omega_2^{-2}\partial_vr_2}(m_2-m_1)\bigg(\frac{1}{r_1}-\frac{1}{r_2}\bigg)\\
&+\big(1+\frac{4}{\omega}\big)\frac{\omega^2}{160\pi}\frac{-\partial_ur_2}{\Omega_2^{-2}\partial_vr_2}(m_2-m_1)\bigg(\frac{1}{r_1}-\frac{1}{r_2}\bigg)\\
=&\bigg(1+\frac{\omega}{4}+\frac{\omega}{20}(4+\omega)\bigg)\frac{-\partial_ur_2}{8\pi\Omega_2^{-2}\partial_vr_2}(m_2-m_1)\bigg(\frac{1}{r_1}-\frac{1}{r_2}\bigg).
\end{align*}
Using the assumption that {\color{black}$\omega<\frac{2}{3}$}, we have $\omega+4<5$, and therefore
\begin{align*}
\Theta^2&\leq\big(1+\frac{\omega}{4}+\frac{5\omega}{20}\big)\frac{-\partial_ur_2}{8\pi\Omega_2^{-2}\partial_vr_2}(m_2-m_1)\bigg(\frac{1}{r_1}-\frac{1}{r_2}\bigg)\\
&=\bigg(1+\frac{\omega}{2}\bigg)\frac{-\partial_ur_2}{8\pi\Omega_2^{-2}\partial_vr_2}(m_2-m_1)\bigg(\frac{1}{r_1}-\frac{1}{r_2}\bigg).
\end{align*}
\end{proof}
\begin{lemma}\label{keylemma2}
	Assume that $\eta\geq\frac{8\epsilon}{\omega}$,  {\color{black}the initial data along $\underline{C}$ is not super-charged,} and $\mathcal{D}(0,v_1)\cup\mathcal{R}$ is free of trapped surfaces. Then
\begin{align*}
\frac{\Omega_2^{-2}\partial_vr_2}{\Omega_1^{-2}\partial_vr_1}(u)\leq e^{-(1-\frac{\omega}{2})\eta(u)}
\end{align*}
for all $u\in[u_0,u_*]$.
\end{lemma}
\begin{proof}
	Dividing both sides of equation $\eqref{EMS3}$ by $\Omega^{-2}\partial_vr$ and integrating from $v_1$ to $v_2$, we get
	\begin{align*}
	\ln|\Omega_2^{-2}\partial_vr_2|-\ln|\Omega_1^{-2}\partial_vr_1| = \ln\bigg(\frac{\Omega_2^{-2}\partial_vr_2}{\Omega_1^{-2}\partial_vr_1}\bigg)=-4\pi\int_{v_1}^{v_2}\frac{r|\partial_v\phi|^2}{\partial_vr}dv
	\end{align*}
	By equation $\eqref{massv}$ and the definition of the Hawking mass, {\color{black}we have}
	
	\begin{align*}
	\frac{1}{r-2m}\big(\partial_vm-\frac{Q^2\partial_vr}{2r^2}\big)=\frac{-8\pi r^2\Omega^{-2}\partial_ur}{-4r\Omega^{-2}\partial_ur\partial_vr}=\frac{2\pi r|\partial_v\phi|^2}{\partial_vr}.
	\end{align*}
	Hence for any $u\in[u_0,u_*]$,
	
	\begin{align*}
	&\ln\bigg(\frac{\Omega_2^{-2}\partial_vr_2}{\Omega_1^{-2}\partial_vr_1}\bigg)=-2\int_{v_1}^{v_2}\frac{1}{r-2m}\big(\partial_vm-\frac{Q^2\partial_vr}{2r^2}\big)dv\\
	&\leq-2\int_{v_1}^{v_2}\frac{1}{r}\big(\partial_vm-\frac{Q^2\partial_vr}{2r^2}\big)dv\leq\frac{-2}{r_2}\int_{v_1}^{v_2}\big(\partial_vm-\frac{Q^2\partial_vr}{2r^2}\big)dv\\
	&=-\frac{2(m_2-m_1)}{r_2}+\frac{2}{r_2}\int_{r_1}^{r_2}\frac{Q^2}{2r^2}dr=-\eta+\frac{2}{r_2}\int_{r_1}^{r_2}\frac{Q^2}{2r^2}dr.
	\end{align*}
	By Lemma $\ref{bound for charge}$, we have
	\begin{align*}
	    \frac{Q(u,v)^2}{r(u,v)^2}\leq\frac{\omega}{2}\frac{2(m(u,v)-m(u,v_1))}{r(u,v)}.
	\end{align*}
	Hence,
	\begin{align*}
	    \frac{2}{r_2}\int_{r_1}^{r_2}\frac{Q^2}{2r^2}dr&\leq\frac{\omega}{2r_2}\int_{r_1}^{r_2}\frac{2(m(u,v)-m(u,v_1))}{r(u,v)}dr\\
	    &\leq{\color{black}\o\cdot}\frac{m_2-m_1}{r_2}\ln\big(\frac{r_2}{r_1}\big)\leq\frac{\omega}{2}\ln\big(\frac{3}{2}\big)\eta\leq\frac{\omega}{2}\eta.
	\end{align*}
Combining the above estimates, we get:
\begin{align*}
    \ln\bigg(\frac{\Omega_2^{-2}\partial_vr_2}{\Omega_1^{-2}\partial_vr_1}\bigg)\leq-\bigg(1-\frac{\omega}{2}\bigg)\eta.
\end{align*}
Exponentiating both sides of the above inequality gives us the desired result.
\end{proof}
\subsection{The proof of Theorem $\ref{thm1.3}$}
{\color{black}
	
	We are finally ready to prove a lower bound on $\frac{d\eta}{du}$. The presence of charge case poses some difficulties not present in the uncharged case. This is because in order to apply Lemma \ref{bound for charge}, \ref{keylemma1} and \ref{keylemma2}, we need to ensure that the $\eta>\frac{8\epsilon}{\omega}$ assumption always hold to get a lower bound on $\frac{d\eta}{du}$. On the other hand, we exactly need this lower bound on $\frac{d\eta}{du}$ to prove the assumption that $\eta>\frac{8\epsilon}{\omega}$. Hence we need to do this using a bootstrap argument again. }\\
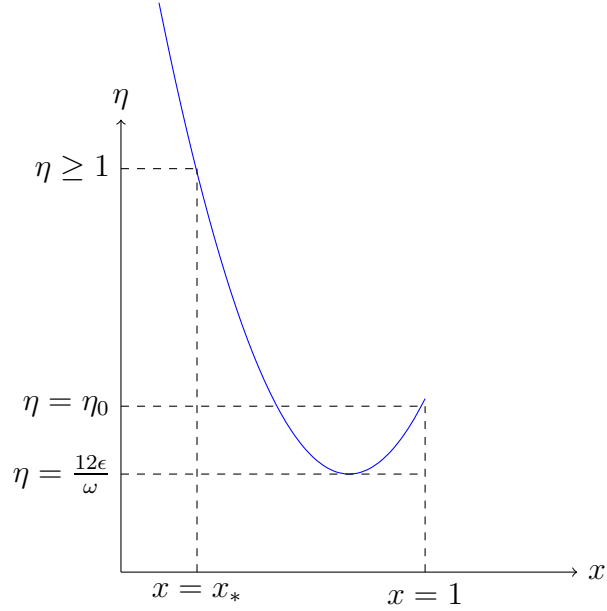
\begin{figure}
    \centering
    \begin{tikzpicture}
      \draw[->] (0,0) -- (6,0) node[right] {$x$};
      \draw[->] (0,0) -- (0,6) node[above] {$\eta$};
      \draw[scale=1,domain=0.5:4,smooth,variable=\x,blue] plot ({\x},{(\x-3)*(\x-3)+1.3});
      \draw [dashed] (4,0)node[anchor = north]{$x=1$} -- (4,2.2);
      \draw [dashed] (0,2.2)node[anchor = east]{$\eta = \eta_0$} -- (4,2.2);
      \draw [dashed] (0,1.3)node[anchor = east]{$\eta = \frac{12\epsilon}{\omega}$} -- (4,1.3);
      \draw [dashed] (1,0)node[anchor = north]{$x=x_*$} -- (1,5.35);
      \draw [dashed] (0,5.35)node[anchor = east]{$\eta \geq 1$} -- (1,5.35);
\end{tikzpicture}
 \caption{Idea of proof of Lemma \ref{theoremlemma} and Theorem \ref{thm1.3}}
 \label{fig:graphcomplex}
\end{figure}

We will in fact prove something a little stronger: we show that $\eta(u)\geq \frac{12\epsilon}{\omega}$ for $u\in[u_0,u_*]$. {\color{black}We use a bootstrap argument to prove this in Lemma $\ref{theoremlemma}$}: Assuming that the differential inequality holds for all $u\in [u_0,u']$, where $u'<u_*$, then $\eta(u)\geq\frac{12\epsilon}{\omega}$ holds in a slightly larger region as well.\\
\begin{lemma}\label{theoremlemma}
	Assume that the region $\mathcal{D}(0,v_1)\cup\mathcal{R}$ is free of trapped surfaces and {\color{black}the initial data along $\underline{C}$ is not super-charged.} Then if $\eta_0\geq\frac{13\epsilon}{\omega}+g_\omega(\delta_0)$, we have  $\eta(x)\geq\frac{12\epsilon}{\omega}$ for all {\color{black}$x(u)\in \big[\frac{3\delta_0}{1+\delta_0},1\big]$}. Over here, $g_\omega(x)$ is defined as:
	\begin{align}\label{g omega x}
	g_\omega(x):=\frac{1+\frac{\omega}{2}}{1-\frac{\omega}{2}}\frac{1}{(1+x)^2}\bigg(\bigg(\frac{2^{1-\frac{\omega}{2}}}{\omega}+\frac{1}{2^{1+\frac{\omega}{2}}(1+\frac{\omega}{2})}\bigg)x^{1-\frac{\omega}{2}}-\frac{2}{\omega}x-\frac{1}{1+\frac{\omega}{2}}x^2\bigg).
	\end{align}
\end{lemma}

\begin{proof}
	We define $u' := \sup\{u\in [u_0,u_*]\big|\eta(s)\geq\frac{12\epsilon}{\omega}\text{ for all } s\in [u_0,{\color{black}u}]\}$. We are going to show that $u' = u_*$, where $x(u_*)=\frac{3\delta_0}{1+\delta_0}$. {\color{black}A sketch of this is provided in figure \ref{fig:graphcomplex}.}
	
	We calculate $\frac{d\eta}{dx}$. The following computations holds for all $u\in [u_0,u_*]$: 
	
	\begin{align}\label{inequalities}
	&\frac{d\eta}{dx}= \frac{d\eta}{du}\bigg/\frac{dx}{du} = \frac{r_2(u_0)}{\partial_ur_2}\bigg(-\frac{2\partial_ur_2}{r_2^2}(m_2-m_1)+\frac{2}{r_2}\partial_u(m_2-m_1)\bigg)\nonumber\\
	=&-\frac{\eta}{x}+\frac{2}{x\partial_ur_2}(-8\pi r_2^2\Omega_2^{-2}\partial_vr_2|D_u\phi_2|^2+8\pi r_1^2\Omega_1^{-2}\partial_vr_1|D_u\phi_1|^2+\frac{Q_2^2\partial_ur_2}{2r_2^2}-\frac{Q_1^2\partial_ur_1}{2r_1^2})\nonumber\\
	{\color{black}\leq}&-\frac{\eta}{x}-\frac{16\pi\partial_vr_2\Omega_2^{-2}}{x\partial_ur_2}\bigg(r_2^2|D_u\phi_2|^2-\frac{\Omega_1^{-2}\partial_vr_1}{\Omega_2^{-2}\partial_vr_2}r_1^2|D_u\phi_1|^2\bigg)+\frac{Q_2^2}{xr_2^2},\nonumber\\
	&{\color{black}\hspace{0.5cm}\text{where we used that } \frac{Q_1^2\partial_ur_1}{xr_1^2\partial_ur_2}\geq 0.}
       \end{align}
	
 	Now we focus our attention on the region $[u_0,u']$. Since $\eta\geq \frac{12\epsilon}{\omega}\geq\frac{8\epsilon}{\omega}$ in $[x',1]$, we can use Lemma $\eqref{keylemma2}$ to bound the factor in the second term:
	\begin{align*}
	r_2^2|D_u\phi_2|^2-\frac{\Omega_1^{-2}\partial_vr_1}{\Omega_2^{-2}\partial_vr_2}r_1^2|D_u\phi_1|^2&\leq r_2^2|D_u\phi_2|^2-e^{\eta(1-\frac{\omega}{2})} r_1^2|D_u\phi_1|^2\\
	&=\Theta^2+2\Theta|D_u\phi_1|r_1+(1-e^{\eta(1-\frac{\omega}{2})})r_1^2|D_u\phi_1|^2.
	\end{align*}
	The last expression, being a quadratic in $\Theta$, can be bounded by a monic quadratic polynomial in $\Theta$:
	
	\begin{align}\label{thetaquadraticbound}
	\Theta^2+2\Theta|D_u\phi_1|r_1+(1-e^{\eta(1-\frac{\omega}{2})})r_1^2|D_u\phi_1|^2&\leq\bigg(1+\frac{1}{e^{\eta(1-\frac{\omega}{2})}-1}\bigg)\Theta^2\nonumber\\
	&\leq\bigg(1+\frac{1}{\eta(1-\frac{\omega}{2})}\bigg)\Theta^2,
	\end{align}
	where we used the fact that $\eta(1-\frac{\omega}{2})\geq 0$ in the second inequality. Then $\eqref{thetaquadraticbound}$ {\color{black}combined with $\eqref{inequalities}$ gives}:
	\begin{align*}
	\frac{d\eta}{dx}\leq-\frac{\eta}{x}-\frac{16\pi\partial_vr_2\Omega_2^{-2}}{x\partial_ur_2}\bigg(1+\frac{1}{\eta(1-\frac{\omega}{2})}\bigg)\Theta^2+\frac{Q_2^2}{xr_2^2}.
	\end{align*}
	Applying Lemma \ref{keylemma1}, we have 
	\begin{align}\label{inequalities2}
	\frac{d\eta}{dx}&\leq-\frac{\eta}{x}+\frac{\eta}{x}\bigg(1+\frac{\omega}{2}\bigg)\bigg(1+\frac{1}{\eta(1-\frac{\omega}{2})}\bigg)\bigg(\frac{r_2}{r_1}-1\bigg)+\frac{Q_2^2}{xr_2^2}.
	\end{align}
	Using Proposition \ref{mixed derivatives of r}, we get
	\begin{align*}
	\delta(u) = \frac{r_2(u)-r_1(u)}{r_2(u)-(r_2(u)-r_1(u))}\leq\frac{r_2(u_0)-r_1(u_0)}{r_2(u)-(r_2(u_0)-r_1(u_0))}=\frac{\delta_0}{x(u)(1+\delta_0)-\delta_0}.
	\end{align*} 
	Combining with $\eqref{inequalities2}$, and using Lemma $\eqref{bound for charge}$ to bound the term involving $Q$, we obtain
	\begin{align*}
	\frac{d\eta}{dx}&\leq\eta\bigg(\big(1+\frac{\omega}{2}\big)\frac{\delta_0}{x^{{\color{black}2}}(1+\delta_0)-{\color{black}x}\delta_0}-\frac{1}{x}\bigg)+\frac{1+\frac{\omega}{2}}{1-\frac{\omega}{2}}{\color{black}\frac{1}{x}}\frac{\delta_0}{x(1+\delta_0)-\delta_0}+\frac{Q_2^2}{xr_2^2}\\
	&\leq\eta\bigg(\big(1+\frac{\omega}{2}\big)\frac{\delta_0}{x^{{\color{black}2}}(1+\delta_0)-{\color{black}x}\delta_0}-\frac{1}{x}\bigg)+\frac{1+\frac{\omega}{2}}{1-\frac{\omega}{2}}{\color{black}\frac{1}{x}}\frac{\delta_0}{x(1+\delta_0)-\delta_0}+\frac{\eta}{x}\frac{\omega}{2}\\
	&=-\frac{\eta}{x}\bigg(1-\frac{\omega}{2}-\big(1+\frac{\omega}{2}\big)\frac{\delta_0}{x(1+\delta_0)-\delta_0}\bigg)+\frac{1+\frac{\omega}{2}}{1-\frac{\omega}{2}}\frac{1}{x}\frac{\delta_0}{x(1+\delta_0)-\delta_0}.
	\end{align*}
	\\
	Defining $g(x):=1-\frac{\omega}{2}-\big(1+\frac{\omega}{2}\big)\frac{\delta_0}{x(1+\delta_0)-\delta_0}$ and $f(x):=\frac{1+\frac{\omega}{2}}{1-\frac{\omega}{2}}\frac{\delta_0}{x(1+\delta_0)-\delta_0}$, we obtain the following differential inequality which holds for all $x\in [x',1]$:
	\begin{align*}
	\frac{d\eta}{dx}+\eta\frac{g(x)}{x}-\frac{f(x)}{x}\leq 0.
	\end{align*}
	To solve this differential inequality, we multiply by an integrating factor {\color{black}to get}:
	\begin{gather*}
	\frac{d}{dx}\bigg(e^{-\int_x^1\frac{g(s)}{s}ds}\eta(x)\bigg)-e^{-\int_x^1\frac{g(s)}{s}ds}\frac{f(x)}{x}\leq 0\\
	\implies\bigg[e^{-\int_{t}^1\frac{g(s)}{s}ds}\eta(t)\bigg]_{t=x}^{t=1}\leq \int_x^1e^{-\int_{t}^1\frac{g(s)}{s}ds}\frac{f}{t}dt+C,
	\end{gather*}
	where $C$ can be chosen to be any value which makes the inequality hold at the initial point $x=1$. Also denote $G(x):= \int_x^1\frac{g(s)}{s}ds$ and $F(x):=\int_x^1e^{-G(s)}\frac{f}{s}ds$. In this notation, we get
	\begin{align*}
	\eta_0-e^{-G(x)}\eta(x)\leq F(x)+C. 
	\end{align*}
	Since $\eta_0 = \eta(x)|_{x=1}$ by definition, and $F(1) = G(1) = 0$, setting $C = 0$ makes the inequality {\color{black}tight}. Hence, in the {\color{black}interval $[x',1]$}, we conclude that the following inequality holds:
	\begin{align}\label{maininequality}
	\eta_0-e^{-G(x)}\eta(x)\leq F(x)
	\end{align}
	
Now {\color{black}we} compute explicit expressions for $G(x)$ and $F(x)$:

\begin{align*}
G(x) &= \int_x^1\frac{1-\frac{\omega}{2}}{s}-\frac{1+\frac{\omega}{2}}{s}\frac{\delta_0}{s(1+\delta_0)-\delta_0}ds\\
&=\int_x^1\frac{1-\frac{\omega}{2}}{s}+\frac{1+\frac{\omega}{2}}{s}-\frac{(1+\frac{\omega}{2})(1+\delta_0)}{s(1+\delta_0)-\delta_0}ds\\
&=\ln\bigg(\frac{s^2}{\big(s(1+\delta_0)-\delta_0\big)^{1+\frac{\omega}{2}}}\bigg)\bigg|^1_x = \ln\bigg(\frac{\big(x(1+\delta_0)-\delta_0\big)^{1+\frac{\omega}{2}}}{x^2}\bigg).
\end{align*}

\begin{align*}
&\,\,F(x)= \int_x^1\frac{s^2}{\big(s(1+\delta_0)-\delta_0\big)^{1+\frac{\omega}{2}}}\frac{f}{s}ds = \int_x^1\frac{1+\frac{\omega}{2}}{1-\frac{\omega}{2}}\frac{\delta_0s}{(s(1+\delta_0)-\delta_0)^{2+\frac{\omega}{2}}}ds\\
&=\frac{1+\frac{\omega}{2}}{1-\frac{\omega}{2}}\frac{\delta_0}{1+\delta_0}\int_x^1\frac{s(1+\delta_0)-\delta_0}{(s(1+\delta_0)-\delta_0)^{2+\frac{\omega}{2}}}+\frac{\delta_0}{(s(1+\delta_0)-\delta_0)^{2+\frac{\omega}{2}}}ds\\
&=\frac{1+\frac{\omega}{2}}{1-\frac{\omega}{2}}\frac{\delta_0}{1+\delta_0}\int_x^1\frac{1}{(s(1+\delta_0)-\delta_0)^{1+\frac{\omega}{2}}}+\frac{\delta_0}{(s(1+\delta_0)-\delta_0)^{2+\frac{\omega}{2}}}ds\\
&=\frac{1+\frac{\omega}{2}}{1-\frac{\omega}{2}}\frac{\delta_0}{(1+\delta_0)^2}\bigg[-\frac{2}{\omega}\frac{1}{\big(s(1+\delta_0)-\delta_0\big)^\frac{\omega}{2}}-\frac{1}{1+\frac{\omega}{2}}\frac{\delta_0}{\big(s(1+\delta_0)-\delta_0\big)^{1+\frac{\omega}{2}}}\bigg]\bigg|_x^1\\
&=\frac{1+\frac{\omega}{2}}{1-\frac{\omega}{2}}\frac{\delta_0}{(1+\delta_0)^2}\bigg(\frac{2}{\omega}\frac{1}{\big(x(1+\delta_0)-\delta_0\big)^\frac{\omega}{2}}-\frac{2}{\omega}+\frac{1}{1+\frac{\omega}{2}}\frac{\delta_0}{\big(x(1+\delta_0)-\delta_0\big)^{1+\frac{\omega}{2}}}-\frac{\delta_0}{1+\frac{\omega}{2}}\bigg).
\end{align*}
Observe that $F(x)$ is monotonically decreasing and hence obtains its maximum at $x=\frac{3\delta_0}{1+\delta_0}$ on the interval $[\frac{3\delta_0}{1+\delta_0},1]$. Therefore,
\begin{align*}
F(x)\leq F\bigg(\frac{3\delta_0}{1+\delta_0}\bigg)&=\frac{1+\frac{\omega}{2}}{1-\frac{\omega}{2}}\frac{\delta_0}{(1+\delta_0)^2}\bigg(\frac{2^{1-\frac{\omega}{2}}}{\omega}\delta_0^{-\frac{\omega}{2}}-\frac{2}{\omega}+\frac{1}{2^{1+\frac{\omega}{2}}(1+\frac{\omega}{2})}\delta_0^{-\frac{\omega}{2}}-\frac{\delta_0}{1+\frac{\omega}{2}}\bigg)\\
&=\frac{1+\frac{\omega}{2}}{1-\frac{\omega}{2}}\frac{1}{(1+\delta_0)^2}\bigg(\bigg(\frac{2^{1-\frac{\omega}{2}}}{\omega}+\frac{1}{2^{1+\frac{\omega}{2}}(1+\frac{\omega}{2})}\bigg)\delta_0^{1-\frac{\omega}{2}}-\frac{2}{\omega}\delta_0-\frac{1}{1+\frac{\omega}{2}}\delta_0^2\bigg)\\
&=g_\omega(\delta_0).
\end{align*}
Substituting the expressions for $F(x)$ and $G(x)$ into \eqref{maininequality}, for all $x\in [x',1]$, we get
{\color{black}\begin{align}\label{gronwall1}
&\eta(x)\geq e^{G(x)}\big(\eta_0-F(x)\big)\nonumber\\
\geq&\frac{\big(x(1+\delta_0)-\delta_0\big)^{1+\frac{\omega}{2}}}{x^2}\nonumber\\
&\quad\cdot\bigg(\eta_0-\frac{1+\frac{\omega}{2}}{1-\frac{\omega}{2}}\frac{1}{(1+\delta_0)^2}\bigg(\bigg(\frac{2^{1-\frac{\omega}{2}}}{\omega}+\frac{1}{2^{1+\frac{\omega}{2}}(1+\frac{\omega}{2})}\bigg)\delta_0^{1-\frac{\omega}{2}}-\frac{2}{\omega}\delta_0-\frac{1}{1+\frac{\omega}{2}}\delta_0^2\bigg)\bigg)\nonumber\\
=&\frac{\big(x(1+\delta_0)-\delta_0\big)^{1+\frac{\omega}{2}}}{x^2}\cdot \big(\eta_0-g_{\o}(\d_0)\big).
\end{align}}
For the last identity, we use the definition of $g_{\o}(x)$ in \eqref{g omega x}. 

{\color{black}Since $\omega<\frac{2}{3}$ implies that $\frac{x^2}{(x(1+\delta_0)-\delta_0)^{1+\frac{\omega}{2}}}$ is monotonically increasing}, we get:
\begin{align*}
\sup_{x\in [x_*,1]}\frac{x^2}{\big(x(1+\delta_0)-\delta_0\big)^{1+\frac{\omega}{2}}} = 1.
\end{align*}
Combining this with the hypothesis that
\begin{align*}
\eta_0>\frac{13\epsilon}{\omega}+g_\omega(\delta_0),
\end{align*}
we {\color{black}obtain} the inequality: 
\begin{align*}
\eta_0 &>\frac{13\epsilon}{\omega}+g_\omega(\delta_0)\geq\frac{13\epsilon}{\omega}\frac{x^2}{\big(x(1+\delta_0)-\delta_0\big)^{1+\frac{\omega}{2}}}+g_\omega(\delta_0)
\end{align*}
for $x\in[x',1]$. Substituting the above into \eqref{gronwall1} gives us:
\begin{align*}
\eta(x)\geq\frac{13\epsilon}{\omega},\hspace{0.5cm}\text{for all } x\in[x',1].
\end{align*}
However, by the continuity of $\eta(x)$, we can find $x''<x'$ such that $\eta(x)\geq \frac{12\epsilon}{\omega}$ for all $x\in [x'',1]$, i.e. $\eta(u)\geq\frac{12\epsilon}{\omega}$ for all $u\in [u_0,u'']$, contradicting the supremum property of $u'$. 
\end{proof} 

We are now ready to prove the main theorem of this paper.
\begin{proof}{\textit{(Theorem \ref{thm1.3})}}
	We prove the theorem by contradiction. Suppose that $\mathcal{R}$ contains no trapped surfaces {\color{black}or MOTS}, in particular $\partial_vr_2>0$ for $u\in [u_0,u_*]$. Then Lemma \ref{keylemma1} applies and $\eqref{maininequality}$ holds for $x\in[x_*,1]$. Rearranging $\eqref{maininequality}$ gives us
	\begin{align*}
	\eta_0\leq e^{-G(x)}\eta(x)+F(x)< e^{-G(x)}+F(x), \text{ for all } x\in[x_*,1]
	\end{align*}
	where we used the assumption that $\eta(x) = \frac{2(m_2-m_1)}{r_2}\leq\frac{2m_2}{r_2}<1$ in the second inequality. In particular, by letting $x = x_*=\frac{3\delta_0}{1+\delta_0}$, we get:
	\begin{align*}
	e^{-G(x_*)}+F(x_*)&\leq \frac{x_*^2}{\big(x_*(1+\delta_0)-\delta_0\big)^{1+\frac{\omega}{2}}}+g_\omega(\delta_0)\\
	&=\frac{9}{2^{1+\frac{\omega}{2}}(1+\delta_0)^2}\delta_0^{1-\frac{\omega}{2}}+g_\omega(\delta_0),
	\end{align*}
	and hence $\eta_0<\frac{9}{2^{1+\frac{\omega}{2}}(1+\delta_0)^2}\delta_0^{1-\frac{\omega}{2}}+g_\omega(\delta_0)$, giving us the desired contradiction.
\end{proof}

\section{Appendix}
\subsection{Trapped Surface Formation for the Einstein Scalar Field}
Here we provide a proof of Christodolou's sharp trapped surface formation criterion as in \cite{Chr.1}. In the case for the real scalar field, the system of equations \eqref{EMS1} to \eqref{EMS9} is reduced to
\begin{gather}
r\partial_v\partial_ur+\partial_vr\partial_ur=-\frac{\Omega^2}{4}\label{UEMS1},\\
\partial_u(\Omega^{-2}\partial_ur) = -4\pi r\Omega^{-2}|\partial_u\phi|^2\label{UEMS2},\\
\partial_v(\Omega^{-2}\partial_vr) = -4\pi r\Omega^{-2}|\partial_v\phi|^2\label{UEMS3},\\
r\partial_u\partial_v\phi+\partial_ur\partial_v\phi+\partial_vr\partial_u\phi=0\label{UEMS4}.
\end{gather}

Also, the derivatives of the Hawking mass become:
\begin{align}
\partial_um= =-8\pi r^2\Omega^{-2}\partial_vr|\partial_u\phi|^2\label{massureal},\\
\partial_vm = -8\pi r^2\Omega^{-2}\partial_ur|\partial_v\phi|^2\label{massvreal}.
\end{align}

For convenience, we restate theorem $\ref{thm1.1}$ here.

\begin{customthm}{1.1}
	Define the function
	\begin{align*}
	E(x):=\frac{x}{(1+x)^2}\bigg[\ln\bigg(\frac{1}{2x}\bigg)+5-x\bigg].
	\end{align*}
	Consider the system (\ref{ES}) with characteristic initial data along $u=u_0$ and $v=v_1$. For initial mass input $\eta_0$ along $u=u_0$,  if the following lower bound holds:
	\begin{align*}
	\eta_0>E(\delta_0),
	\end{align*}
	then a trapped surface {\color{black}$S_{u,v}$, with properties $\partial_v r(u,v)<0$} and $\partial_u r(u,v)< 0$, {\color{black}forms} in the region $[u_0,u_*]\times[v_1,v_2]\subset\mathcal{R}$.
\end{customthm}

In this section, we will first give a few technical estimates to the dynamical quantities in the strip $[u_0,0]\times[v_1,v_2]$. These will be used in the proof for Theorem \ref{thm1.1}.
We start off by showing that $\partial_ur$ is negative and bounded away from 0.
\begin{lemma}\label{negativeincoming}
	$\partial_ur\leq -\frac{1}{2}\Omega^2$ everywhere in {\color{black}$\mathcal{D}(0,v_1)\cup\big([u_0,0]\times[v_1,\infty)\big)$}
	\begin{proof}
		Rewrite $\eqref{UEMS1}$ as
		\begin{align*}
		\partial_v\big(r\partial_ur\big)=-\frac{\Omega^2}{4}.
		\end{align*}
		\noindent {\color{black}Note that $\Omega^2=1$ along ${\color{black}C}$.}	Integrating both sides from $0$ to $v$ and dividing by $r$:
		\begin{align}
		-\frac{v}{4r(u_0,v)} = \partial_ur(u_0,v)\label{partialur}.
		\end{align}
		Setting $v = 0 $ in the above gives us
		\begin{align}\label{boundoutgoing}
		-\frac{1}{4\partial_vr(u_0,0)}=\partial_ur(u_0,0) = -\partial_vr(u_0,0)\implies\partial_vr(u_0,0)=\frac{1}{2}.
		\end{align}
		
		\noindent	Since $\Omega^2 = 1$ on $C$ as well, $\eqref{UEMS3}$ gives us that $\partial_v\partial_vr\leq 0$, i.e. $r$ is concave with respect to $v$. Combining this with the fact that $r(u_0,0) = v(u_0,0) = 0$, we have: 
		\begin{align*}
		\frac{r}{v}(u_0,v)\leq\partial_vr(u_0,0).
		\end{align*}
		
		\noindent	Hence, 
		\begin{align*}
		\frac{r}{v}(u_0,v)\leq\frac{1}{2}.
		\end{align*}
		
		\noindent	Substitute this into \eqref{partialur}, we get
		\begin{align*}
		\partial_ur(u_0,v)\leq -\frac{1}{2}.
		\end{align*}
		
		\noindent	By $\eqref{UEMS3}$, $\Omega^{-2}\partial_ur$ is decreasing along incoming null geodesics. Hence for a general point in {\color{black}$\mathcal{D}(0,v_1)\cup\big([u_0,0]\times[v_1,\infty)\big)$}, we have $\Omega^{-2}\partial_ur\leq-\frac{1}{2}$.
	\end{proof}
\end{lemma}

\begin{remark}
	$m(u,v)\geq 0$ for all {\color{black}$(u,v)\in \mathcal{D}(0,v_1)\cup\big([u_0,0]\times[v_1,\infty)\big)$}
\end{remark}
\begin{proof}
	
	Given any point $(u,v)\in \mathcal{R}$, we can extend the outgoing null geodesic backwards until it intersects $\Gamma$ at some coordinate $(u,v_c)$, so that $r(u,v_c) = 0$. Using $\eqref{massvreal}$, we have
	\begin{align*}
	\partial_vm = -8\pi r^2\Omega^{-2}\partial_ur|\partial_v\phi|^2,
	\end{align*}
	
	\noindent	and since $\partial_ur\leq 0$ by Lemma $\ref{negativeincoming}$, we get that $\partial_vm\geq 0$. Combining with the fact that $m(u,v_c) = 0$, we obtain the desired result. 
\end{proof}
Next, we show that the mixed derivative of $r$ is always negative. This places a upper bound on the growth on the ratio $\frac{r_2}{r_1}$.
\begin{proposition}\label{mixed derivatives of r real}
	Assume that $\mathcal{D}(0,v_1)\cup\mathcal{R}$ is free of trapped surfaces. Then i) $\partial_u\partial_vr \leq 0$ in $\mathcal{R}$ and ii) $\delta(x):=\frac{r_2}{r_1}-1\leq\frac{1}{2}$ for $u\in[u_0,u_*]$.
\end{proposition}
\begin{proof}
	We rewrite \eqref{UEMS1} into the following equivalent form: 
	\begin{align}\label{UEMS1mass}
	\partial_u\partial_vr=-\frac{\Omega^2}{{\color{black}2}}\frac{m}{r^2}.
	\end{align}
	Since $m\geq 0$, the right side of the above equation is {\color{black}non-positive}. This proves the first part of the lemma.\\
	
	Integrating with respect to $u$, we get:
	
	\begin{align*}
	\partial_vr(u)-\partial_vr(u_0){\color{black}\leq 0}\implies\partial_vr(u)\leq\partial_vr(u_0).
	\end{align*}
	Integrating the above inequality with respect to $v$,
	\begin{align*}
	r_2(u)-r_1(u)\leq r_2(u_0)-r_1(u_0), \text{for all } u\in [u_0,u_*].
	\end{align*}
	Hence, we can use the above inequality to compute a bound for $\delta(u)$:
	\begin{equation}\label{delta inequality}
	\begin{split}
	\delta(u) = \frac{r_2}{r_1}-1&=\frac{r_2-r_1}{r_2-(r_2-r_1)}\leq\frac{r_2(u_0)-r_1(u_0)}{r_2(u)-(r_2(u_0)-r_1(u_0))}\\
	&\leq\frac{\delta_0}{\frac{r_2(u)}{r_1(u_0)}-\delta_0}=\frac{\delta_0}{x(u)(1+\delta_0)-\delta_0},\hspace{0.5cm}\text{ for all }u\in[u_0,u_*],
	\end{split}
	\end{equation}
	where $x(u):=r_2(u)/r_2(u_0)$.	{\color{black}Recall $r_2(u_*):=\f{3\delta_0}{1+\delta_0}\cdot r_2(u_0)$ and} since $x(u)$ is monotonically decreasing, we have 
	\begin{align*}
	x(u)\geq x(u_*) = \frac{3\delta_0}{1+\delta_0} \text{ for }u\in [u_0,u_*],
	\end{align*}
	and hence
	\begin{align*}
	\delta(u)\leq\frac{\delta_0}{2\delta_0} = \frac{1}{2}\text{ for all }u\in [u_0,u_*].
	\end{align*}
\end{proof} 

Next we prove two key lemmas. In the first one we  bound the difference in $r\partial_u\phi$ between $v=v_1$ and $v=v_2$. Then in the second we bound the ratio of $\partial_vr$ between $v=v_1$ and $v=v_2$.

\begin{lemma}\label{keylemma1real}
	Define {\color{black}$\Theta:=r_2\partial_u\phi_2-r_1\partial_u\phi_1$}. Suppose that $\mathcal{D}(0,v_1)\cup\mathcal{R}$ is free of trapped surfaces. Then
	\begin{align*}
	\Theta(u)^2\leq \frac{\partial_ur_2}{8\pi\Omega_2^{-2}\partial_vr_2}(m_2-m_1)\bigg(\frac{1}{r_2}-\frac{1}{r_1}\bigg)(u)
	\end{align*}
	for all $u\in [u_0,u_*]$.
\end{lemma}
\begin{proof}
	We can write the wave equation \eqref{UEMS4} as 
	\begin{align*}
	\partial_v(r\partial_u\phi) = -\partial_ur\partial_v\phi.
	\end{align*}
	By integrating the above equation, we get
	{\color{black}\begin{align}
		\Theta^2 &= \big(r_2\partial_u\phi_2-r_1\partial_u\phi_1\big)^2\nonumber\\&=\bigg|\int_{v_1}^{v_2}-\partial_ur\partial_v\phi dv\bigg|^2\nonumber\leq\bigg(\int_{v_1}^{v_2}-\partial_ur|\partial_v\phi| dv\bigg)^2\nonumber\\
		&\leq\frac{1}{8\pi}\int_{v_1}^{v_2}-8\pi r^2\partial_ur\Omega^{-2}|\partial_v\phi|^2dv\cdot\int_{v_1}^{v_2}-\frac{\partial_ur}{r^2\Omega^{-2}}dv\label{lemma1eqn1real}
		\end{align}}
	where we have applied Holder's inequality for the last inequality.
	
	The first integral can be written in terms of the hawking mass:
	\begin{align}
	\int_{v_1}^{v_2}-8\pi r^2\partial_ur\Omega^{-2}|\partial_v\phi|^2dv&=\int_{v_1}^{v_2}\partial_vm\text{ }dv\nonumber\\
	&=m_2-m_1\label{lemma1eqn2real}.
	\end{align}
	
	To bound the second integral, we apply Proposition $\ref{mixed derivatives of r real}$ to get $\partial_v\partial_ur\leq 0$, and hence $\partial_ur\geq\partial_ur_2$. Also, equation {\color{black}$\eqref{UEMS3}$} implies that $\Omega_2^{-2}\partial_vr_2\leq\Omega^{-2}\partial_vr$. Combining these two pieces of information, we have
	{\color{black}
		\begin{align}
		\int_{v_1}^{v_2}-\frac{\partial_ur}{r^2\Omega^{-2}}dv &=\int_{r_1}^{r_2}-\frac{\partial_ur}{r^2\Omega^{-2}\partial_vr}dr\nonumber\leq-\partial_ur_2\int_{r_1}^{r_2}\frac{1}{r^2\Omega^{-2}\partial_vr}dr\nonumber\\
		&\leq\frac{-\partial_ur_2}{\Omega_2^{-2}\partial_vr_2}\int_{r_1}^{r_2}\frac{1}{r^2}dr=\frac{\partial_ur_2}{\Omega_2^{-2}\partial_vr_2}\big(\frac{1}{r_2}-\frac{1}{r_1}\big)\label{lemma1eqn3real}
		\end{align}}
	Substituting $\eqref{lemma1eqn2real}$ and $\eqref{lemma1eqn3real}$ back into $\eqref{lemma1eqn1real}$ gives us the desired result.		 
\end{proof}
\begin{lemma}\label{keylemma2real}
	Assume that $\mathcal{D}(0,v_1)\cup\mathcal{R}$ is free of trapped surfaces. Then
	\begin{align*}
	\frac{\Omega_2^{-2}\partial_vr_2}{\Omega_1^{-2}\partial_vr_1}(u)\leq e^{-\eta(u)}
	\end{align*}
	for all $u\in[u_0,u_*]$.
\end{lemma}
\begin{proof}
	Dividing both sides of equation {\color{black}$\eqref{UEMS3}$} by $\Omega^{-2}\partial_vr$ and integrating from $v_1$ to $v_2$, we get
	\begin{align*}
	\ln|\Omega_2^{-2}\partial_vr_2|-\ln|\Omega_1^{-2}\partial_vr_1| = \ln\bigg(\frac{\Omega_2^{-2}\partial_vr_2}{\Omega_1^{-2}\partial_vr_1}\bigg)=-4\pi\int_{v_1}^{v_2}\frac{r|\partial_v\phi|^2}{\partial_vr}dv.
	\end{align*}
	By equation $\eqref{massvreal}$ and the definition of the Hawking mass, {\color{black}we have}
	{\color{black}
		\begin{align*}
		\frac{\partial_vm}{r-2m}=\frac{-8\pi r^2\Omega^{-2}\partial_ur|\partial_v\phi|^2}{-4r\Omega^{-2}\partial_ur\partial_vr}=\frac{2\pi r|\partial_v\phi|^2}{\partial_vr}.
		\end{align*}}
	
	Hence for any $u\in[u_0,u_*]$
	
	\begin{align*}
	\ln\bigg(\frac{\Omega_2^{-2}\partial_vr_2}{\Omega_1^{-2}\partial_vr_1}\bigg)&=-2\int_{v_1}^{v_2}\frac{\partial_vm}{r-2m}dv\leq-2\int_{v_1}^{v_2}\frac{1}{r}\partial_vm\text{ }dv\\
	&\leq\frac{-2}{r_2}\int_{v_1}^{v_2}\partial_vm\text{ }dv=-\frac{2(m_2-m_1)}{r_2}=-\eta
	\end{align*}
	Exponentiating both sides of the above inequality gives us the desired result.
\end{proof}
Now we are ready to prove Theorem \ref{thm1.1}.
\begin{proof} (\textit{Theorem \ref{thm1.1}})
	We consider the dimensionless length scale $x(u):=\frac{r_2(u)}{r_2(u_0)}$. Note that $x$ decreases as $u$ increases and $x(u_0) = 1$. {\color{black}We will show that $\frac{d\eta}{dx}$ is bounded from above, i.e. $\frac{d\eta}{du}$ is bounded from below, and hence obtain a lower bound for $\eta(u_*)$.  If this lower bound is greater than $1$, this implies $S(u_*,v_2)$ is a trapped surface, for}
		
		\begin{align*}
		\eta(u_*)= \frac{2(m_2-m_1)}{r_2}(u_*)>1&\implies \frac{2m_2}{r_2}(u_*)>1\\
		&\implies S(u_*,v) \text{ is a trapped surface.}
		\end{align*} See Figure \ref{fig:graphreal} for an illustration. \\
		\begin{figure}
			\centering
			\begin{tikzpicture}
			\draw[->] (0,0) -- (6,0) node[right] {$x$};
			\draw[->] (0,0) -- (0,6) node[above] {$\eta$};
			\draw[scale=1,domain=0.5:4,smooth,variable=\x,blue] plot ({\x},{0.2*(\x-5)*(\x-5)+2});
			\draw [dashed] (4,0)node[anchor = north]{$x=1$} -- (4,2.2);
			\draw [dashed] (0,2.2)node[anchor = east]{$\eta = \eta_0$} -- (4,2.2);
			\draw [dashed] (1,0)node[anchor = north]{$x=x_*$} -- (1,5.35);
			\draw [dashed] (0,5.35)node[anchor = east]{$\eta \geq 1$} -- (1,5.35);
			\end{tikzpicture}
			\caption{Idea of Proof of Theorem \ref{thm1.1}}
			\label{fig:graphreal}
		\end{figure}
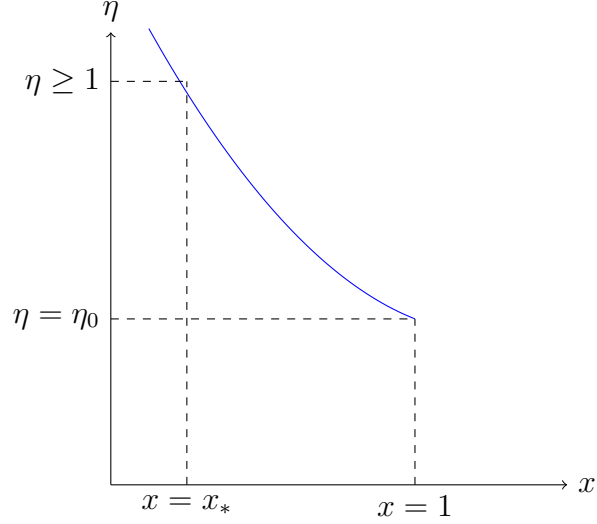
	
	To be precise, we prove a Gronwall-like inequality under the assumption that there is no trapped surface formed before $u_*$. In particular, we  assume that $\partial_vr_2(u)>0$ for all $u\in [u_0,u_*]$. We show that this assumption will lead to a contradiction.\\
	
	Assuming that $\partial_vr_2(u)>0$ for all $u\in [u_0,u_*]$, the following chain of {\color{black}identities} hold in the region $[u_0,u_*]\times [v_1,v_2]$:
	\begin{align}
	\frac{d\eta}{dx} &= \frac{d\eta}{du}\bigg/\frac{dx}{du} = \frac{r_2(u_0)}{\partial_ur_2}\bigg(-\frac{2\partial_ur_2}{r_2^2}(m_2-m_1)+\frac{2}{r_2}\partial_u(m_2-m_1)\bigg)\nonumber\\
	&=-\frac{\eta}{x}+\frac{2}{x\partial_ur_2}(-8\pi r_2^2\Omega_2^{-2}\partial_vr_2|\partial_u\phi_2|^2+8\pi r_1^2\Omega_1^{-2}\partial_vr_1|\partial_u\phi_1|^2)\nonumber\\
	&=-\frac{\eta}{x}-\frac{16\pi\partial_vr_2\Omega_2^{-2}}{x\partial_ur_2}\bigg(r_2^2|\partial_u\phi_2|^2-\frac{\Omega_1^{-2}\partial_vr_1}{\Omega_2^{-2}\partial_vr_2}r_1^2|\partial_u\phi_1|^2\bigg)\label{inequalitiesreal}.
	\end{align}
	
	Using Lemma $\ref{keylemma2real}$, we can bound the factor in the second term:
	\begin{align*}
	r_2^2|\partial_u\phi_2|^2-\frac{\Omega_1^{-2}\partial_vr_1}{\Omega_2^{-2}\partial_vr_2}r_1^2|\partial_u\phi_1|^2&\leq r_2^2|\partial_u\phi_2|^2-e^\eta {\color{black}r_1^2}|\partial_u\phi_1|^2\\
	&=\Theta^2+2\Theta\partial_u\phi_1r_1+(1-e^\eta)r_1^2|\partial_u\phi_1|^2.
	\end{align*} 
	
	The last expression, being a quadratic in $\Theta$, can be bounded by a monic quadratic polynomial in $\Theta$:
	
	\begin{align*}
	\Theta^2+2\Theta\partial_u\phi_1r_1+(1-e^\eta)r_1^2|\partial_u\phi_1|^2&\leq\bigg(1+\frac{1}{e^\eta-1}\bigg)\Theta^2\leq\bigg(1+\frac{1}{\eta}\bigg)\Theta^2,
	\end{align*}
	since $\eta\geq 0$. This last inequality combines with $\eqref{inequalitiesreal}$ to give:
	
	\begin{align*}
	\frac{d\eta}{dx}\leq-\frac{\eta}{x}-\frac{16\pi\partial_vr_2\Omega_2^{-2}}{x\partial_ur_2}\bigg(1+\frac{1}{\eta}\bigg)\Theta^2.
	\end{align*}
	
	Applying Lemma $\ref{keylemma1real}$, we have 
	\begin{align}
	\frac{d\eta}{dx}&\leq -\frac{\eta}{x}-\frac{2}{x}\bigg(1+\frac{1}{\eta}\bigg)\bigg(\frac{1}{r_2}-\frac{1}{r_1}\bigg)(m_2-m_1)\nonumber\\
	&=-\frac{\eta}{x}+\frac{\eta}{x}\bigg(1+\frac{1}{\eta}\bigg)\bigg(\frac{r_2}{r_1}-1\bigg)\label{inequalities2real}.
	\end{align}
	Using {\color{black}(\ref{delta inequality})} we get
	\begin{align*}
	\delta = \frac{r_2-r_1}{r_2-(r_2-r_1)}\leq\frac{r_2(u_0)-r_1(u_0)}{r_2(u)-(r_2(u_0)-r_1(u_0))}=\frac{\delta_0}{x(1+\delta_0)-\delta_0}.
	\end{align*} Combining {\color{black}the above} with $\eqref{inequalities2real}$, we obtain
	\begin{align*}
	\frac{d\eta}{dx}\leq\eta\bigg(\frac{\delta_0}{x^{{\color{black}2}}(1+\delta_0)-{\color{black}x}\delta_0}-\frac{1}{x}\bigg)+\frac{\delta_0}{x^{{\color{black}2}}(1+\delta_0)-{\color{black}x}\delta_0}.
	\end{align*}
	Defining $g(x):=1-\frac{\delta_0}{x(1+\delta_0)-\delta_0}$ and $f(x):=\frac{\delta_0}{x(1+\delta_0)-\delta_0}$, we obtain the following differential inequality:
	\begin{align*}
	\frac{d\eta}{dx}+\eta\frac{g(x)}{x}-\frac{f(x)}{x}\leq 0.
	\end{align*}
	To solve this differential inequality, we multiply by an integrating factor {\color{black} then integrate with respect to $x$}:
	\begin{gather*}
	\frac{d}{dx}\bigg(e^{-\int_x^1\frac{g(s)}{s}ds}\eta(x)\bigg)-e^{-\int_x^1\frac{g(s)}{s}ds}\frac{f(x)}{x}\leq 0\\
	\implies\bigg[e^{-\int_{{\color{black}x'}}^1\frac{g(s)}{s}ds}\eta(x')\bigg]_{x'=x}^{x'=1}\leq \int_x^1e^{-\int_{x'}^1\frac{g(s)}{s}ds}\frac{f}{x'}dx'.
	\end{gather*}
	We denote $G(x):= \int_x^1\frac{g(s)}{s}ds$ and $F(x):=\int_x^1e^{-G(x')}\frac{f}{x'}dx'$. In this notation, we get
	\begin{align*}
	\eta_0-e^{-G(x)}\eta(x)\leq F(x) \implies\eta(x)\geq e^{G(x)}\big(-F(x)+\eta_0\big).
	\end{align*}
	Hence, in the region $[u_0,u_*]\times [v_1,v_2]$ free of trapped surfaces, we conclude that the following inequality holds:
	\begin{align}\label{maininequalityreal}
	\eta(x)\geq e^{G(x)}(-F(x)+\eta_0).
	\end{align}
	
	{\color{black}Now, we} compute explicit expressions for $G(x)$ and $F(x)$:
	
	\begin{align*}
	G(x) &= \int_x^1\frac{1}{s}-\frac{1}{s}\frac{\delta_0}{s(1+\delta_0)-\delta_0}ds\\
	&=\int_x^1\frac{1}{s}+\frac{1}{s}-\frac{1+\delta_0}{s(1+\delta_0)-\delta_0}ds\\
	&=\ln\bigg(\frac{s^2}{s(1+\delta_0)-\delta_0}\bigg)\bigg|^1_x = \ln\bigg(\frac{x(1+\delta_0)-\delta_0}{x^2}\bigg)
	\end{align*}
	
	\begin{align*}
	F(x) &= \int_x^1\frac{s^2}{s(1+\delta_0)-\delta_0}\frac{f}{s}ds = \int_x^1\frac{\delta_0s}{(s(1+\delta_0)-\delta_0)^2}ds\\
	&=\frac{\delta_0}{1+\delta_0}\int_x^1\frac{1}{s(1+\delta_0)-\delta_0}+\frac{\delta_0}{(s(1+\delta_0)-\delta_0)^2}ds\\
	&=\frac{\delta_0}{(1+\delta_0)^2}\ln\bigg(s(1+\delta_0)-\delta_0\bigg)\bigg|_x^1 - \frac{\delta_0^{{\color{black}2}}}{(1+\delta_0)^2}\frac{1}{s(1+\delta_0)-\delta_0}\bigg|_x^1\\
	&=\frac{\delta_0}{(1+\delta_0)^2}\bigg(\ln\big(\frac{1}{x(1+\delta_0)-\delta_0}\big)+\delta_0\big(\frac{1}{x(1+\delta_0)-\delta_0}-1\big)\bigg).
	\end{align*} 
	Using the assumption that there is no trapped surface or MOTS, we have $\eta(x)=\frac{2(m_2-m_1)}{r_2}\leq\frac{2m_2}{r_2}<1$ for $x\in [\frac{3\delta_0}{1+\delta_0},1]$. Rearranging $\eqref{maininequalityreal}$ {\color{black}results in}
	\begin{align*}
	\eta_0\leq e^{-G(x)}\eta(x)+F(x)<e^{-G(x)}+F(x),\text{ for all }x\in \bigg[\frac{3\delta_0}{1+\delta_0},1\bigg].
	\end{align*}
	In particular, we can substitute $x = \frac{3\delta_0}{1+\delta_0}$ into the above equation and get 
	\begin{align*}
	\eta_0< E(\delta_0) = \frac{\delta_0}{(1+\delta_0)^2}\bigg[\log\big(\frac{1}{2\delta_0}\big)+5-\delta_0\bigg].
	\end{align*}
	This gives us the desired contradiction. 
\end{proof}

\subsection{A Special Case of Minkowskian incoming characteristic initial data} 
Prescribe Minkowskian data along $v=v_1$, we can improve the lower bound required on $\eta_0$ in Theorem \ref{thm1.1}. 
\begin{customthm}{1.2}
	Assume that {\color{black}Minkowskian data are prescribed along $v=v_1$ and require $\phi(u,v_1)=0$}. Suppose that the following lower bound on $\eta_0$ holds:
\begin{align*}
\eta_0>\f92\delta_0,
\end{align*}
then there exist a MOTS {\color{black}or a trapped surface} in $[u_0,u_*]\times[v_1,v_2]\subset\mathcal{R}$, i.e. $\partial_vr\leq 0$ at some point in $[u_0,u_*]\times[v_1,v_2]$.
\end{customthm}

\begin{proof}\textit{(Theorem \ref{thm1.2})}
	In this special case, we have $\phi_1\equiv 0$ and $m_1 \equiv 0$. Equation \eqref{inequalitiesreal} now reads: 
	\begin{align*}
	\frac{d\eta}{dx} = -\frac{\eta}{x}-\frac{16\pi\partial_vr_2\Omega_2^{-2}}{x\partial_ur_2}r_2^2|\partial_u\phi_2|^2,
	\end{align*}
	and we also have
	\begin{align*}
	\Theta^2 = r_2^2|\partial_u\phi_2|^2.
	\end{align*}
	Combining the above equations, followed by applying Lemma \ref{keylemma1real}, we get:
	\begin{align*}
	\frac{d\eta}{dx} = -\frac{\eta}{x}-\frac{16\pi\partial_vr_2\Omega_2^{-2}}{x\partial_ur_2}\Theta^2&\leq-\frac{\eta}{x}-\frac{2}{x}\bigg(\frac{1}{r_2}-\frac{1}{r_1}\bigg)(m_2-m_1)\\
	&=-\frac{\eta}{x}+\frac{\eta}{x}\bigg(\frac{r_2}{r_1}-1\bigg)\\
	&\leq -\frac{\eta}{x}+\frac{\eta}{x}\frac{\delta_0}{x(1+\delta_0)-\delta_0}.
	\end{align*}
	Integrating the above inequality:
	\begin{align*}
	&\int_x^1\frac{1}{\eta}d\eta\leq\int_x^1-\frac{1}{s}+\frac{1}{s}\frac{\delta_0}{s(1+\delta_0)-\delta_0}ds=\int_x^1-\frac{2}{s}+\frac{1+\delta_0}{s(1+\delta_0)-\delta_0}ds\\
	&\implies \ln\bigg(\frac{\eta_0}{\eta(x)}\bigg)\leq\ln\bigg(\frac{x^2}{x(1+\delta_0)-\delta_0}\bigg)\\
	&\implies \eta_0\leq\eta(x)\frac{x^2}{x(1+\delta_0)-\delta_0}.
	\end{align*}
	Under the assumption of no trapped surfaces or MOTS, we have $\eta(x)<1$ for all $x\in[\frac{3\delta_0}{1+\delta_0},1]$, hence
	\begin{align*}
	\eta_0\leq \frac{x^2}{x(1+\delta_0)-\delta_0}{\color{black}\text{, for all }x\in\bigg[\frac{3\delta_0}{1+\delta_0},1\bigg].}
	\end{align*}
	In particular, choosing $x=\frac{3\delta_0}{1+\delta_0}$ we have
	\begin{align*}
	\eta_0\leq \f92\delta_0.
	\end{align*}
	{\color{black}This gives us} the desired contradiction to the hypothesis.   
\end{proof}

\end{document}